\documentclass[11pt]{article}

\usepackage{amsfonts}
\usepackage{enumitem}
\usepackage{shuffle,yfonts,amsmath,amssymb,amsthm}
\usepackage{hyperref}
\hypersetup{pdfstartview={XYZ null null 1.00}}
\usepackage{pdfpages}

\newcommand\nc{\newcommand}

\let\Im\relax

\DeclareMathOperator\Im{{{Im}}}
\nc{\ES}{\mathsf{ES}}
\nc{\MZV}{\mathsf{MZV}}
\nc{\CMZV}{\mathsf{CMZV}}
\nc{\MtV}{\mathsf{MtV}}
\nc{\AMtV}{\mathsf{AMtV}}
\nc{\MTV}{\mathsf{MTV}}
\nc{\MSV}{\mathsf{MSV}}
\nc{\MMV}{\mathsf{MMV}}
\nc{\MMVo}{\mathsf{MMVo}}
\nc{\MMVe}{\mathsf{MMVe}}
\nc{\AMMV}{\mathsf{AMMV}}
\nc{\sha}{\shuffle}
\nc{\si}{\sigma}
\nc{\gd}{\delta}
\nc{\ola}{\overleftarrow}
\nc{\ora}{\overrightarrow}
\nc{\lra}{\longrightarrow}
\nc{\Lra}{\Longrightarrow}
\nc\Res{{\rm Res}}
\nc\setX{{\mathsf{X}}}
\nc\fA{{\mathfrak{A}}}
\nc\evaM{{\texttt{M}}}
\nc\evaML{{\text{\em{\texttt{M}}}}}
\nc\z{{\texttt{z}}}
\nc\ta{{\texttt{a}}}
\nc\ty{{\texttt{y}}}
\nc\tx{{\texttt{x}}}
\nc\td{{\texttt{d}}}
\nc\tz{{\texttt{z}}}
\nc\txp{{\tx_1}} 
\nc\txn{{\tx_{-1}}} 
\nc\neo{{1}}
\nc{\yi}{{1}}
\nc\one{{-1}}
\nc\om{{\omega}}
\nc\tom{{\tilde{\omega}}}
\nc\omn{\omega_{-1}}
\nc\omz{\omega_0}
\nc\omp{\omega_{1}}
\nc\eps{{\varepsilon}}
\nc{\bfp}{{\bf p}}
\nc{\bfq}{{\bf q}}
\nc{\bfu}{{\bf u}}
\nc{\bfv}{{\bf v}}
\nc{\bfw}{{\bf w}}
\nc{\bfy}{{\bf y}}
\nc{\bfga}{{\boldsymbol{\sl{\alpha}}}}
\nc{\bfe}{{\boldsymbol{\sl{e}}}}
\nc{\bfi}{{\boldsymbol{\sl{i}}}}
\nc{\bfj}{{\boldsymbol{\sl{j}}}}
\nc{\bfk}{{\boldsymbol{\sl{k}}}}
\nc{\bfl}{{\boldsymbol{\sl{l}}}}
\nc{\bfm}{{\boldsymbol{\sl{m}}}}
\nc{\bfn}{{\boldsymbol{\sl{n}}}}
\nc{\bfs}{{\boldsymbol{\sl{s}}}}
\nc{\bfr}{{\boldsymbol{\sl{r}}}}
\nc{\bft}{{\boldsymbol{\sl{t}}}}
\nc{\bfx}{{\boldsymbol{\sl{x}}}}
\nc{\bfz}{{\boldsymbol{\sl{z}}}}
\nc\bfmu{{\boldsymbol \mu}}
\nc\bfgl{{\boldsymbol \lambda}}
\nc\bfsi{{\boldsymbol \sigma}}
\nc\bfet{{\boldsymbol \eta}}
\nc\bfeta{{\boldsymbol \eta}}
\nc\bfeps{{\boldsymbol \varepsilon}}
\nc\bfone{{\bf 1}}

\catcode`!=11
\let\!int\int \def\int{\displaystyle\!int}
\let\!lim\lim \def\lim{\displaystyle\!lim}
\let\!sum\sum \def\sum{\displaystyle\!sum}
\let\!sup\sup \def\sup{\displaystyle\!sup}
\let\!inf\inf \def\inf{\displaystyle\!inf}
\let\!cap\cap \def\cap{\displaystyle\!cap}
\let\!max\max \def\max{\displaystyle\!max}
\let\!min\min \def\min{\displaystyle\!min}
\let\!frac\frac \def\frac{\displaystyle\!frac}
\catcode`!=12

\let\oldsection\section
\renewcommand\section{\setcounter{equation}{0}\oldsection}

\allowdisplaybreaks

\DeclareMathOperator{\Li}{Li}

\nc\UU{\mbox{\bfseries U}}
\nc\FF{\mbox{\bfseries \itshape F}}
\nc\h{\mbox{\bfseries \itshape h}}\nc\dd{\mbox{d}}
\nc\g{\mbox{\bfseries \itshape g}}
\nc\xx{\mbox{\bfseries \itshape x}}

\nc\gl{{\lambda}}
\nc\gD{{\Delta}}
\nc\tlga{{\tilde{\alpha}}}
\nc\ga{{\alpha}}
\nc\gb{{\beta}}
 \nc{\gam}{{\gamma}}
 \nc{\gG}{{\Gamma}}
 \nc{\vep}{{\varepsilon}}
 \nc{\gs}{{\sigma}}
 \nc{\gth}{{\theta}}
 \nc{\tlgth}{\tilde{\theta}}
 \nc{\gr}{{\rho}}
 \nc{\gL}{{\Lambda}}
 \nc{\gS}{{\Sigma}}
 \nc{\gf}{{\varphi}}
 \nc{\gk}{{\kappa}}
 \nc{\gm}{{\mu}}
 \nc{\gM}{{M}}
 \nc{\gz}{{\zeta}}
 \nc{\tlg}{{\tilde{g}}}
 \nc{\tllg}{{\tilde{\tilde{g}}}}
 \nc{\tlG}{{\tilde{G}}}
 \nc{\tlh}{{\tilde{h}}}
 \nc{\tllh}{{\tilde{\tilde{h}}}}
 \nc{\tlH}{{\tilde{H}}}
\nc{\tle}{{\tilde{e}}}
 \nc{\tlle}{{\tilde{\tilde{e}}}}
 \nc{\tlE}{{\tilde{E}}}
 \nc{\tlgz}{{\tilde{\zeta}}}
 \nc{\gO}{{\Omega}}
 \nc{\sif}{{\mathcal S}}
 \nc{\gt}{{\tau}}
 \nc{\bt}{{\chi}}
 \nc{\tlt}{{\tilde{t}}}
 \nc{\tlgk}{{\tilde{\gk}}}
 \nc{\binn}{{\binom{2n}{n}}}
 \nc{\bimm}{{\binom{2m}{m}}}

\def\N{\mathbb{N}}

\def\Q{\mathbb{Q}}

\def\ze{\zeta}

\def\xx{\left(\frac{1-x}{1+x} \right)}

\nc\divg{{\text{div}}}
\theoremstyle{plain}
\newtheorem{thm}{Theorem}[section]

\newtheorem{cor}[thm]{Corollary}

\theoremstyle{definition}

\newtheorem{re}[thm]{Remark}
\newtheorem{ex}[thm]{Example}

\nc{\myone}{{1}}
\nc{\myO}{{\mathsf O}}
\nc{\myL}{{\mathsf L}}

\begin{document}
\title{\bf Some Variants of Ap\'{e}ry-Type Series and Level Four Colored Multiple Zeta Values}
\author{
{Ce Xu${}^{a,}$\thanks{Email: cexu2020@ahnu.edu.cn,  ORCID 0000-0002-0059-7420.}\ \ and Jianqiang Zhao${}^{b,}$\thanks{Email: zhaoj@ihes.fr, corresponding author, ORCID 0000-0003-1407-4230.}}\\[1mm]
\small a. School of Mathematics and Statistics, Anhui Normal University, Wuhu 241002, PRC\\
\small b. Department of Mathematics, The Bishop's School, La Jolla, CA 92037, USA}

\date{}
\maketitle

\noindent{\bf Abstract.} In this paper, we study Ap\'{e}ry-type series involving the central binomial coefficients
\begin{align*}
\sum_{n_1>\cdots>n_d>0} \frac1{4^{n_1}}\binom{2n_1}{n_1} \frac{1}{n_1^{s_1}\cdots n_d^{s_d}}
\end{align*}
and its variations where the summation indices may have mixed parities and some or all ``$>$'' are replaced by ``$\ge$'', as long as the series are defined. These sums have naturally appeared in the calculation of massive Feynman integrals by the work of Jegerlehner, Kalmykov and Veretin. We show that all these sums can be expressed as $\mathbb Q$-linear combinations of the real and/or imaginary parts of the colored multiple zeta values at level four, i.e., special values of multiple polylogarithms at fourth roots of unity. We also show that the corresponding series where ${\binom{2n_1}{n_1}}/4^{n_1}$ is replaced by ${\binom{2n_1}{n_1}}^2/16^{n_1}$ can be expressed in a similar way except for a possible extra factor of $1/\pi$.


\section{Introduction}
Set  $a_0=a_0(x)=1$, $a_n(x)=\frac1{4^{n}}\binom{2n}{n}x^{2n}$ and $a_n=a_n(1)$ for all $n\in\N$.
In the first part of this series \cite{XuZhao2022a}, we considered the following Ap\'{e}ry-type inverse binomial series
\begin{align*}
\sum_{n_1>\cdots>n_d>0}   \frac{a_{n_1}^{-1}}{n_1^{s_1}\cdots n_d^{s_d}} \quad
\end{align*}
for all positive integers $s_1\ge 2, s_2,\dots,s_d\ge 1$, and their variants with some or all of
$n_j$'s replaced by $2n_j\pm 1$ and some or all of ``$>$'' replaced by ``$\ge$'', as long
as the series are defined. By generalizing Akhilesh's ideas in \cite{Akhilesh1,Akhilesh}
we showed that the above sum can be expressed as $\Q$-linear combinations of the real and/or
the imaginary parts of some colored multiple zeta values of level 4, i.e., multiple polylogarithms evaluated at 4th roots of unity.
We also proved similar results after replacing $a_{n_1}^{-1}$ by $a_{n_1}^{-2}$.

In this paper, we will turn to another class of Ap\'{e}ry-type series defined by
\begin{align}\label{defn-gaAll}
\sum_{n_1>\cdots>n_d>0}  \frac{a_{n_1}}{n_1^{s_1}\cdots n_d^{s_d}} \quad\text{and}\quad
\sum_{n_1>\cdots>n_d>0}  \frac{a_{n_1}^2}{n_1^{s_1}\cdots n_d^{s_d}}
\end{align}
and similar variants with the summation indices having mixed parities.
We call these \emph{Ap\'{e}ry-type (central) binomial series} (not inverse type).
Instances of this type already appeared in Leshchiner's work \cite{Leshchiner} extending Ap\'ery's formula.
These sums have also appeared in the calculation of massive Feynman integrals \cite{JegerlehnerKV2003}.

Recall that for any positive integers $s_1,\dots,s_d$ and $N$th roots of unity $z_1,\dotsc,z_d$ the
\emph{colored multiple zeta values} (CMZVs) of level $N$ are defined by
\begin{equation}\label{equ:defnMPL}
\Li_{\bfs}(\bfz):=\sum_{n_1>\cdots>n_d>0}
\frac{z_1^{n_1}\dots z_d^{n_d}}{n_1^{s_1} \dots n_d^{s_d}},
\end{equation}
which converge if $(s_1,z_1)\ne (1,1)$ (see \cite{Racinet2002} and \cite[Ch. 15]{Zhao2016}).
Starting from the 1990s,
mathematicians and theoretical physicists have been attracted to the study of
colored multiple zeta values due to their frequent and
sometimes unexpected appearance in quite a few different branches of mathematics and physics,
in particular, the computation of many Feyman diagrams (see, e.g., \cite{BognerLu2013,Broadhurst1996,Broadhurst1999,BroadhurstKr1997,Brown2011,Duhr2015,Kreimer2015,Panzer2014a,ParkerSSV2015,Todorov2014}).

The CMZVs can be expressed using Chen's iterated integrals (see \cite[Sec.~2.1]{Zhao2016}):
\begin{align*}
 \Li_{\bfs}(\bfz)=\int_0^1 \ta^{s_1-1} \tx_{a_1}\cdots \ta^{s_d-1} \tx_{a_d},
\end{align*}
where $\ta=dt/t$, $\tx_a=dt/(a-t)$, and $a_j=1/(z_1\cdots z_j)$ for all $j=1,\dots,d$. The theory of iterated integrals was developed firstly by K.T. Chen in the 1960's. It has played important roles in the study of algebraic topology and algebraic geometry in the past half century. Its simplest form is
\begin{align*}
\int_0^1 f_1(t)dtf_{2}(t)dt\cdots f_p(t)dt
=&\, \int_0^1 f_1(t)dt\circ f_{2}(t)dt\circ \cdots \circ f_p(t)dt \\
:=&\, \int\limits_{1>t_1>\cdots>t_p>0}f_1(t_1)f_{2}(t_{2})\cdots f_p(t_p)dt_1dt_2\cdots dt_p.
\end{align*}
One can extend these to iterated integrals over any piecewise smooth path on the complex plane via pull-backs.
We refer the interested reader to Chen's original work \cite{KTChen1971,KTChen1977} for more details.

The key idea to compute the infinite sums \eqref{defn-gaAll} and their odd-indexed variations is to derive some recursive formulas for the ``tails'' of a variable version of these series, which are given by the following (see also \eqref{an-it1}--\eqref{an-it4} and \eqref{equ-gbE-depthd}).

\bigskip
\noindent
\textbf{Recursive Formulas.} For all $n\ge0$, $s\ge 1$, and $y\in(-\pi/2,\pi/2)$, we have
\begin{align*}
\sum_{m>n} \frac{a_m(\sin y)}{(2m)^s} =&\, \int_0^y (\cot t\,dt)^{s-1} (1-\csc t \,dt\circ\sec t)a_n(\sin t) \tan t\,dt, \\
\sum_{m>n} \frac{a_m(\sin y)}{(2m+1)^s} =&\,\csc y \int_0^y (\cot t\,dt)^{s-1} (1-dt\circ\csc t\sec t)a_n(\sin t) \sin t\tan t\,dt,\\
\sum_{m\ge n} \frac{a_m(\sin y)}{(2m+1)^s} =&\,\csc y\int_0^y (\cot t\,dt)^{s-1} (\csc t-dt\circ\sec t) a_n(\sin t) \tan t\,dt,\\
\sum_{m>n} \frac{a_m(\sin y)}{2m-1} =&\,\cos y\int_0^y a_n(\sin t)  \tan t \sec t \, dt,  \\
\sum_{m>n} \frac{a_m(\sin y)}{(2m-1)^s}
=&\,  \sin y \int_0^y (\cot t\,dt)^{s-2} (\cot^2 t\,dt) a_n(\sin t)  \tan t \sec t \, dt \quad (s\ge2).
\end{align*}
Here we have used a notation generalizing that for the iterated integrals.
See the beginning of the next section for details.

We remark that a particular type of odd variations already appeared implicitly
in \cite[(A.25)]{JegerlehnerKV2003}. Indeed, it is easy to
verify that $(2n)a_n=(2n-1)a_{n-1}$ so that for any function $S$
\begin{align}\label{equ-NeedOddVar}
 \frac12 \sum_{n\ge0}\frac{a_n}{(n+1)^c}S(n)=\sum_{n\ge 1}\frac{a_{n-1}}{(2n)n^{c-1}}S(n-1)
 =\sum_{n>0}\frac{a_n}{(2n-1)n^{c-1}}S(n-1).
\end{align}

We now summarize the content of this paper.
In the next three sections, we will repeatedly apply the recursive formulas in the above to
obtain the iterated integral expressions of three general Ap\'ery-type binomial series. In Section 5,
we will analyze these iterated integrals carefully to prove the main result of this paper relating these
sums to the CMZVs of level 4. We then apply a beta integral to derive the iterated integral expressions
for the corresponding Ap\'ery-type series in which $a_n^2$ appears. In the last section, we
answer some questions in our previous work \cite{XuZhao2021b} and provide a few enlightening examples.
We have computed many other examples and attach them as two appendices to this paper.

\section{Ap\'ery-type central binomial series}
In this section, we will consider the Ap\'{e}ry-type series defined in \eqref{defn-gaAll}.
We will need to extend Chen's iterated integrals by combining 1-forms and functions as follows.
For any $r\in\N$, 1-forms $f_1(t)\,dt,\dots,f_{r+1}(t)\,dt$ and functions $F_1(t),\dots,F_r(t)$,
 $G_1(t),\dots,G_r(t)$, we extend the definition of iterated integrals by setting recursively
\begin{align*}
&\, \int_a^b \bigg( \prod_{j=1}^{r}{}\hskip-2.35ex{\scriptstyle \circ}\hskip1ex \Big[F_j(t)+ f_j(t)\,dt\circ G_j(t)\Big] \bigg)\circ f_{r+1}(t)\,dt\\
:=&\, \int_a^b  \bigg( \prod_{j=1}^{r-1}{}\hskip-2.35ex{\scriptstyle \circ}\hskip1ex \Big[F_j(t)+ f_j(t)\,dt\circ G_j(t)\Big] \bigg)\circ\big(F_r(t)f_{r+1}(t)\big)\,dt\\
+&\, \int_a^b \bigg( \prod_{j=1}^{r-1}{}\hskip-2.35ex{\scriptstyle \circ}\hskip1ex \Big[F_j(t)+ f_j(t)\,dt\circ G_j(t)\Big] \bigg)\circ f_r(t)\,dt \circ G_r(t)f_{r+1}(t)\,dt,
\end{align*}
where
$$
\prod_{j=1}^{r}{}\hskip-2.35ex{\scriptstyle \circ}\hskip1ex \ga_j=\ga_1\circ\cdots\circ\ga_r.
$$

Our first result concerns the tails of the Ap\'ery-type series \eqref{defn-gaAll}. To be consistent with all the major theorems
of this paper, we formulate it as an even-indexed variation. For $s\in\N$ we define
\begin{align}
p_s(t):=&\, \tan t\,dt \bigg(\frac{dt}{\tan t}\bigg)^{s-1} ( 1-\csc t \,dt\circ\sec t). \label{defn-p}
\end{align}

\begin{thm} \label{thm-ga}
For all $n\in\N_0$, $\bfs=(s_1,\dots,s_d)\in\N^d$ and $y\in[-\pi/2,\pi/2]$ we have
\begin{align*}
\ga(\bfs;\sin y)_n:=&\, \sum_{n_1>\cdots > n_d>n} \frac{a_{n_1}(\sin y)}{(2n_1)^{s_1}\cdots (2n_d)^{s_d}} \\
=&\,   \cot y \frac{d}{dy} \int_0^y p_{s_1} \circ\cdots \circ p_{s_d} \circ a_{n}(\sin t) \tan t\,dt.
\end{align*}
Here if $y=\pm \pi/2$ then the right-hand side of the above is defined to be the limit as $y\to \pm \pi/2$.
\end{thm}

We will again call the sum $|\bfs|:=s_1+\dots+s_d$ the weight and $d$ the depth of the series $\ga(\bfs;\sin y)$, respectively.

\begin{proof} The proof is in the same spirit as that of \cite[Thm.~4]{Akhilesh}. Define
\begin{align*}
u_n(y)=\tan y\sum_{m>n} \frac1{4^m} \bimm \frac{\sin^{2m-1}y}{2m-1},
\quad
v_n(y)=\int_0^y \binn\frac{\sin^{2n-1}t}{4^n} \tan^2 t\, dt.
\end{align*}
We claim that $u_n(y)=v_n(y)$. Indeed, first we have
\begin{align*}
 \frac{d}{dy}\Big(u_{n-1}(y)-u_n(y)\Big)=&\, \frac1{4^n}\binn \frac{d}{dy} \frac{ \sin^{2n-1} y \tan y}{(2n-1) } \\
=&\, \binn \frac{\sin^{2n-1} y}{4^n} \cdot \frac{ (2n-1)+\sec^2 y}{2n-1}
\end{align*}
and
\begin{align*}
 \frac{d}{dy}\Big(v_{n-1}(y)-v_n(y)\Big)
=&\, \frac{\sin^{2n-3} y}{4^n} \tan^2 t \left(4 \binom{2n-2}{n-1} - \binn \sin^2 y\right)\\
=&\, \frac{\sin^{2n-1} y}{4^n} \binn \sec^2 y \left(\frac{2n}{2n-1} - \sin^2 y\right)\\
=&\, \frac{\sin^{2n-1} y}{4^n} \binn \sec^2 y \left(\frac{1}{2n-1} +\cos^2 y\right).
\end{align*}
Hence for all $n\ge 1$
\begin{align*}
u_{n-1}(y)-u_n(y) = v_{n-1}(y)-v_n(y)
\end{align*}
since clearly the two sides both vanish when $y=0$.
Further, for all $y\in(-\pi/2,\pi/2)$,
\begin{align*}
\lim_{n\to \infty} u_n(y) = \lim_{n\to \infty} v_n(y)=0
\end{align*}
since by Stirling's formula
\begin{align*}
 \frac1{4^n}\binn \sim \frac1{\sqrt{\pi n}} \quad \text{as } n\to \infty.
\end{align*}
By telescoping we see immediately that $u_n(y)=v_n(y)$, or equivalently,
\begin{align}\label{equ-gam-d=1}
\sum_{m>n} \frac{1}{4^m} \bimm \frac{\sin^{2m-1} y}{2m-1} =\cot y\int_0^y \binn \frac{\sin^{2n-1}t}{4^n} \tan^2 t \, dt.
\end{align}
Differentiating \eqref{equ-gam-d=1} we get
\begin{align}\label{strictGr}
\sum_{m>n} a_{m}\sin^{2m-2} y \cos y= \binn \frac{\sin^{2n} y}{4^n} \sec y
-\csc^2 y\int_0^y \binn \frac{\sin^{2n}t}{4^n} \tan t \sec t\, dt.
\end{align}
Multiplying this by $\sin y$ and integrating, we have
\begin{align*}
\sum_{m>n} \frac{1}{4^m} \bimm \frac{\sin^{2m} y}{2m} =&\, \int_0^y (1-\csc t \,dt\circ\sec t)\binn \frac{\sin^{2n}t}{4^n} \tan t\,dt.
\notag
\end{align*}
Repeatedly multiplying these by $\cot y$ and integrating, we have for all $s\ge 1$
\begin{align}\label{equ-ga-depthd}
\sum_{m>n} \frac{1}{4^m} \bimm \frac{\sin^{2m} y}{(2m)^s} =&\, \int_0^y\bigg(\frac{dt}{\tan t}\bigg)^{s-1} (1-\csc t \,dt\circ\sec t)\binn \frac{\sin^{2n}t}{4^n} \tan t\,dt.
\end{align}
The general depth cases of the theorem follow readily from the above depth one cases by iteration. We leave the details
to the interested reader.
\end{proof}

\begin{ex} If $s_1=\cdots=s_d=1$ then we see that
\begin{align*}
\ga(1;\sin y):=&\, \sum_{n>0} \frac1{4^{n}} \binn \frac{\sin^{2n} y}{2n}
= \int_0^y ( \csc t-\cot t) \,dt \\
=&\, \left. \log \frac{\csc t}{\csc t+\cot t} \right|_0^y
= \log 2- \log(1+\cos y).
\end{align*}
Thus $\ga(1;1)=\log 2$. When $d=2$ we see that
\begin{align*}
\ga(1,1;\sin y):=&\, \int_0^y \Big( \csc t \,dt \circ\sec t-\cot t \,dt \Big) ( \csc t-\cot t) \,dt \\
=&\,\int_0^y
 ( \csc t-\cot t) \,dt ( \csc t-\cot t) \,dt+ \csc t \,dt \circ (\sec t-1)( \csc t-\cot t) \,dt.
\end{align*}
Note that
\begin{align*}
&\, \int_0^u (\sec t-1)( \csc t-\cot t) \, dt=\int_0^u\frac{(1-\cos t)^2 \,dt}{\sin t\cos t}\\
=&\,\int_0^u \frac{(\cos t-1) \,d\cos t}{(1+\cos t)\cos t}
= \int_0^u \left(\frac{2}{1+\cos t}-\frac{1}{\cos t} \right) d\cos t
= \log\frac{(1+\cos u)^2}{4\cos u}.
\end{align*}
Setting $a=\cos y$ we have
\begin{align*}
\ga(1,1;\sin y):=&\, \frac12\Big(\int_0^y (\csc t-\cot t) \,dt\Big)^2 + \int_0^y \log\frac{(1+\cos t)^2}{4\cos t} \csc t \,dt \\
=&\, \frac12 \log^2 \frac{2}{1+a}+  \int_0^y - \frac{1}{1-\cos^2 t}\log\frac{(1+\cos t)^2}{4\cos t} d\cos t \\
=&\, \frac12 \log^2 \frac{2}{1+a}+  \int_a^1 \frac{1}{1-u^2}\log\frac{(1+u)^2}{4u} du \\
=&\, \frac12 \log^2 \frac{2}{1+a}+ \frac12\Li_2(1-a)- \Li_2\Big(\frac{1-a}2\Big)+\frac12\Li_2(-a) \\
&\, + \log2\log(1+a) + \frac12\log(a)\log(1 + a) - \frac12\log^2 2 - \frac12\log^2(1 + a) + \frac14\zeta(2) \\
=&\, \frac12\Li_2(1-a)- \Li_2\Big(\frac{1-a}2\Big)+\frac12\Li_2(-a)+\frac12\log(a)\log(1 + a) + \frac14\zeta(2).
\end{align*}
Applying the identity $\Li_2(t)+\Li_2(1-t)+\log t\log(1-t)=\zeta(2)$ with $t=1/2$ we see that
\begin{align*}
\ga(1,1;1):=&\, \frac12\log^2 2 + \frac14\zeta(2)\approx 0.65146002.
\end{align*}
\end{ex}

\section{An odd variant}
In this section, we turn to the Ap\'ery-type series similar to \eqref{defn-gaAll} but with
summation indices restricted to odd numbers only. Set
\begin{align}
\theta_s(t)=&\, \tan t\,dt \bigg(\frac{dt}{\tan t}\bigg)^{s-1} ( 1-dt\circ\csc t \, \sec t), \label{defn-theta}\\
q_s(t)=&\, \sec t\,dt \bigg(\frac{dt}{\tan t}\bigg)^{s-1} (\csc t-dt\circ\sec t) .\label{defn-q}
\end{align}

\begin{thm} \label{thm-ga-gb}
For all $n\in\N_0$, $\bfs=(s_1,\dots,s_d)\in\N^d$ and $y\in[-\pi/2,\pi/2]$ we have
\begin{align*}
\gb(\bfs;\sin y)_n:=&\, \sum_{n_1> \cdots > n_d>n} \frac{a_{n_1}(\sin y)}{(2n_1+1)^{s_1}\cdots (2n_d+1)^{s_d}} \\
=&\, \csc y\cot y \frac{d}{dy} \int_0^y \theta_{s_1} \circ\cdots \circ \theta_{s_d} \circ a_{n}(\sin t) \sin t \tan t\,dt,\\
\gb^\star(\bfs;\sin y)_n:=&\, \sum_{n_1\ge \cdots \ge n_d\ge n}\frac{a_{n_1}(\sin y)}{(2n_1+1)^{s_1}\cdots (2n_d+1)^{s_d}} \\
=&\, \cot y \frac{d}{dy} \int_0^y q_{s_1} \circ\cdots \circ q_{s_d} \circ a_{n}(\sin t) \tan t\,dt.
\end{align*}
Here when $y=\pm \pi/2$ we understand the right-hand side of the above as the limits $y\to \pm \pi/2$.
\end{thm}

\begin{proof} Recall that from \eqref{strictGr} we get
\begin{align}\label{strictGr2}
\sum_{m>n} a_{m}\sin^{2m-2} y \cos y= \binn \frac{\sin^{2n} y}{4^n} \sec y
-\csc^2 y\int_0^y \binn \frac{\sin^{2n}t}{4^n} \tan t \sec t\, dt.
\end{align}
Multiplying this by $\sin^2 y$ and integrating, we have
\begin{align*}
\sum_{m>n} \frac{1}{4^m} \bimm \frac{\sin^{2m+1} y}{2m+1} =&\, \int_0^y (1-dt\circ\csc t\sec t) \binn \frac{\sin^{2n+1}t}{4^n} \tan t\,dt.
\end{align*}
Repeatedly multiplying this by $\cot y$ and integrating, we have for all $s\ge 1$
\begin{align}\label{equ-gb-depthd}
\sum_{m>n} \frac{1}{4^m} \bimm \frac{\sin^{2m+1} y}{(2m+1)^s} =&\, \int_0^y \bigg(\frac{dt}{\tan t}\bigg)^{s-1} (1-dt\circ\csc t\sec t) \binn \frac{\sin^{2n+1}t}{4^n} \tan t\,dt.
\end{align}

On the other hand, multiplying \eqref{strictGr2} by $\sin^2 y$
and adding $\frac{1}{4^n} \binn \sin^{2n} y \cos y$ on both sides we get
\begin{align*}
\sum_{m\ge n} \frac{1}{4^m} \bimm \sin^{2m} y \cos y= \binn \frac{\sin^{2n}y}{4^n} \sec y
-\int_0^y \binn \frac{\sin^{2n}t}{4^n} \tan t \sec t\, dt.
\end{align*}
Integrating we have
\begin{align*}
\sum_{m\ge n} \frac{1}{4^m} \bimm \frac{\sin^{2m+1} y}{2m+1} =&\, \int_0^y (\csc t-dt\circ\sec t) \binn \frac{\sin^{2n+1}t}{4^n} \sec t\,dt.
\end{align*}
Repeatedly multiplying these by $\cot y$ and integrating, we have for all $s\ge 2$
\begin{align}
\sum_{m\ge n} \frac{1}{4^m} \bimm \frac{\sin^{2m+1} y}{(2m+1)^s} =&\, \int_0^y \bigg(\frac{dt}{\tan t}\bigg)^{s-1}(\csc t-dt\circ\sec t) \binn \frac{\sin^{2n+1}t}{4^n} \sec t\,dt. \label{equ-gbE-depthd}
\end{align}
The general cases of the theorem follow readily from the above depth one cases by iteration.
\end{proof}

\section{Another odd variant}
In this section, we study another variant of the Ap\'ery-type series similar to \eqref{defn-gaAll}. Even though these series look very similar to those in the previous section, they are fundamentally different and therefore require different approaches to handle them. Set
\begin{align}\label{defn-k}
\gf_s(t)=
\left\{
 \begin{array}{ll}
 \tan t\, dt, \quad & \hbox{if $s=1$;} \\
 \tan^2 t\, dt \circ (\cot t\, dt)^{s-2}\circ\cot^2 t\, dt, \quad & \hbox{if $s\ge 2$.}
 \end{array}
\right.
\end{align}

\begin{thm} \label{thm-gam-Trig}
For all $n\in\N_0$, $\bfs=(s_1,\dots,s_d)\in\N^d$ and $y\in[-\pi/2,\pi/2]$ we have
\begin{align*}
\gam(\bfs;\sin y)_n:=&\, \sum_{n_1>n_2> \cdots > n_d>n}
\frac{a_{n_1}(\sin y)}{(2n_1-1)^{s_1}\cdots (2n_d-1)^{s_d}} \\
=&\, \cos y \cot y \frac{d}{dy} \int_0^y \gf_{s_1}\circ\cdots \circ\gf_{s_d}\circ a_{n}(\sin t) \tan t \sec t \,dt.
\end{align*}
Here when $y=\pm \pi/2$ we understand the right-hand side of the above as the limits $y\to \pm \pi/2$.
\end{thm}

\begin{proof}
Recall from \eqref{equ-gam-d=1} we have
\begin{align*}
\sum_{m>n} \frac{1}{4^m} \bimm \frac{\sin^{2m-1} y}{2m-1} =\cot y\int_0^y \binn \frac{\sin^{2n-1}t}{4^n} \tan^2 t \, dt.
\end{align*}
For any $s\ge 2$ we can repeatedly multiply this equation by $\cot y$ and integrate to get
\begin{align}\label{equ-gam-depthd}
\sum_{m>n} \frac{1}{4^m} \bimm \frac{\sin^{2m-1} y}{(2m-1)^s}
=\int_0^y \bigg(\frac{dt}{\tan t}\bigg)^{s-2} \frac{dt}{\tan^2 t}\binn \frac{\sin^{2n-1}t}{4^n} \tan^2 t \, dt.
\end{align}
This proves the case of $d=1$ in the theorem. The general depth case now follows easily by repeating this procedure.
We leave the details to the interested reader.
\end{proof}

\begin{ex}\label{exa-gam1} If $s_1=\cdots=s_d=1$ then we see that
\begin{align*}
\gam(\bfs;\sin y):=&\, \sum_{n_1>\cdots > n_d>0}
\frac{a_{n_1} (\sin y)}{(2n_1-1)\cdots (2n_d-1)} \\
=&\, \cos y \int_0^y (\tan t \,dt)^{d-1} \circ\tan t \sec t \,dt
=1-\cos y \sum_{j=1}^{d-1} \frac{\log^j |\sec y|}{j!}.
\end{align*}
In particular, we see that $\gam(1_d;1)=1$ for all $d\ge 1$.
\end{ex}

It turns out that the evaluation $\gam(1_d;1)=1$ is just a very special case of the following result on
the tails $\gam(1_d;1)_n$, which in turn easily leads to Thm.~\ref{thm-LeadGammaBlock}.

\begin{cor} \label{cor-Chi-Trig}
For all $n\in\N_0$ and $d\in\N$ we have
\begin{align*}
\gam(1_d;1)_n=&\, \sum_{n_1>n_2> \cdots > n_d>n} \frac{a_{n_1}}{(2n_1-1) \cdots (2n_d-1) } =a_{n}.
\end{align*}
\end{cor}
\begin{proof} For any $y\in(-\pi/2,\pi/2)$ and $d\in\N$, we set
\begin{align*}
f_d(y)=\int_0^y (\tan t \,dt)^{d-1} \circ a_{n}(\sin t)\tan t \sec t \,dt.
\end{align*}
Then by Thm.~\ref{thm-gam-Trig} we get
\begin{align*}
\gam(\bfs;1)_n=&\, \lim_{y\to (\pi/2)^-} \frac{f_d(y)}{\sec y}= \lim_{y\to (\pi/2)^-} \frac{f'_d(y)}{\sec y\tan y}
\end{align*}
by L'H\^opital's rule. But $f_d(y)=(\tan y ) f_{d-1}(y)$ if $d\ge 2$. Thus repeatedly applying L'H\^opital's rule we finally get
\begin{align*}
\gam(1_d;1)_n=&\, \lim_{y\to (\pi/2)^-} \frac{f_1(y)}{\sec y}= \lim_{y\to (\pi/2)^-} \frac{f'_1(y)}{\sec y\tan y}
=\lim_{y\to (\pi/2)^-} a_{n}(\sin y)=a_n,
\end{align*}
as desired.
\end{proof}

\section{Mixed parities}
We now turn to the general Ap\'ery-type series whose summation indices can have mixed parities.
For any $r\in\N$, 1-forms $f_1(t)\,dt,\dots,f_r(t)\,dt$ and function $F(t)$, we define
\begin{align*}
f_1(t)\,dt \circ \ola{F(t)}  \circ f_2(t)\,dt = &\, F(t)f_1(t)\,dt \circ f_2(t)\,dt,\\
\int_0^y \ola{F(t)}  \circ f_1(t)\,dt  \circ \cdots \circ f_r(t)\,dt=&\, F(y)\int_0^y   f_1(t)\,dt  \circ \cdots \circ f_r(t)\,dt.
\end{align*}

\begin{thm}\label{thm-Mixed-Parity}
Suppose $d\in \N$ and $\bfs=(s_1,\dots,s_d)\in\N^d$. Let $y\in(-\pi/2,\pi/2)$.
Set $F_s= (\cot t\, dt)^{s-1}$ and
\begin{align*}
\gl_{2n,s}(t)=&\, F_s (\tan t\,dt-\csc t \,dt\, d(\sec t) ) , \\
\gl_{2n+1,s}(t)=&\,  \ola{\csc t}\circ  F_s (\sin t\tan t\,dt-dt\, d(\sec t) ),\\
\gl_{2n-1,1}(t)=&\,   \ola{\cos t}\circ  \tan t \sec t \, dt,\\
\gl_{2n-1,s}(t)=&\,   \ola{\sin t}\circ F_{s-1} \frac{dt}{\tan^2 t}   \tan t \sec t \, dt \quad (s\ge2) .
\end{align*}
Then for any $l_1(n)=2n,2n\pm 1$ and if $l_2(n),\dots,l_d(n)=2n,2n+1$ we have the tails
\begin{align}\label{equ-Mixed-Parity-Trig}
& \sum_{n_1 > \cdots> n_d> n}
\frac{a_{n_1}(\sin y)}{ l_1(n_1)^{s_1}\cdots l_d(n_d)^{s_d}}
 = \int_0^y \gl_{l_1,s_1} \circ \cdots \gl_{l_d,s_{d}} \circ \ola{a_n(\sin t)} .
\end{align}
Using non-trigonometric 1-forms, we get for all $x\in(-1,1)$
\begin{align}\label{equ-Mixed-Parity-NonTrig}
& \sum_{n_1 > \cdots> n_d> n} \frac{a_{n_1}(x)}{ l_1(n_1)^{s_1}\cdots l_d(n_d)^{s_d}}
 =   \int_0^x \gL_{l_1,s_1}  \circ \cdots \gL_{l_d,s_{d}} \circ  \ola{a_n(t)} ,
\end{align}
where $\gL_{l,s}$ are obtained from $\gl_{l,s}$ by applying the substitution $t\to\sin^{-1} t$.
\end{thm}

\begin{proof}
By \eqref{equ-ga-depthd},
\eqref{equ-gb-depthd} and  \eqref{equ-gam-depthd}, we readily get
\begin{align}
\sum_{m>n} \frac{a_m(\sin y)}{(2m)^s} =&\, \int_0^y F_s (1-\csc t \,dt\circ\sec t)a_n(\sin t) \tan t\,dt,\label{an-it1}\\
\sum_{m>n} \frac{a_m(\sin y)}{(2m+1)^s} =&\,\csc y \int_0^y F_s (1-dt\circ\csc t\sec t)a_n(\sin t) \sin t\tan t\,dt. \label{an-it2}\\
\sum_{m>n} \frac{a_m(\sin y)}{2m-1} =&\,\cos y\int_0^y a_n(\sin t)  \tan t \sec t \, dt, \label{an-it3}\\
\sum_{m>n} \frac{a_m(\sin y)}{(2m-1)^s}
=&\,  \sin y \int_0^y F_{s-1} \frac{dt}{\tan^2 t} a_n(\sin t)  \tan t \sec t \, dt \quad (s\ge2).\label{an-it4}
\end{align}
Thus
\begin{align*}
\gl_{2n,s}(t)=&\, F_s\Big(\tan t\,dt-\csc t \,dt\, d(\sec t) \Big )= F_s ( \csc t \,dt\circ \sec t-\cot t\,dt ) , \\
\gl_{2n+1,s}(t)=&\,  \ola{\csc t}\circ F_s\Big (\sin t\tan t\,dt-dt\, d(\sec t) \Big )
=\ola{\csc t}\circ F_s (dt\circ \sec t-\cos t\,dt ),\\
\gl_{2n-1,1}(t)=&\,     \ola{\cos t}\circ  d(\sec t) ,\\
\gl_{2n-1,s}(t)=&\,     \ola{\sin t}\circ F_{s-1} \frac{dt}{\tan^2 t} d(\sec t)
    =\ola{\sin t}\circ F_{s-1} \bigg(\frac{dt}{\sin t\tan t}- \frac{dt}{\tan^2 t}\circ \sec t\bigg) \quad (s\ge2)   .
\end{align*}

For convenience, we call the right-hand side of \eqref{an-it1}
(resp.~\eqref{an-it2}, resp. \eqref{an-it3} and \eqref{an-it4})
a $\ga$-block (resp. $\gb$-block, resp. $\gam$-block). In \eqref{equ-Mixed-Parity-Trig},
each $s_j$ corresponds to (a variation of) such a block. We find that
after starting with a block in \eqref{an-it1}-\eqref{an-it4},
we can  repeatedly applying \eqref{an-it1}-\eqref{an-it4} to insert
all the middle blocks until the end.
This concludes the constructive proof of the theorem.
\end{proof}

\begin{thm}\label{thm-LeadGammaBlock}
Suppose $d\in \N$, $\bfs=(s_1,\dots,s_d)\in\N^d$. Let $l_1(n),\dots,l_d(n)=2n,2n\pm 1$.
Then
\begin{align*}
\sum_{n_0> n_1\succ n_2 \succ\dots \succ n_d\succ 0}
 \frac{a_{n_0}}{(2n_0-1) l_1(n_1)^{s_1}\cdots l_d(n_d)^{s_d}}
=\sum_{n_1\succ n_2 \succ\dots \succ n_d\succ 0}
 \frac{a_{n_1}}{l_1(n_1)^{s_1}\cdots l_d(n_d)^{s_d}},
\end{align*}
where  ``$\succ$'' can be either ``$\ge$'' or ``$>$'', provided the series is defined.
\end{thm}

\begin{proof} This follows from Cor.~\ref{cor-Chi-Trig} immediately.
\end{proof}

Due to Thm.~\ref{thm-LeadGammaBlock}, we can always assume that if the leading block is a $\gam$-block then $s_1\ge 2$.

\begin{thm}\label{thm-Mixed-Parity-all}
Let $d\in \N$, $\bfs=(s_1,\dots,s_d)\in\N^d$. Let $l_1(n),\dots,l_d(n)=2n,2n\pm 1$.
Set $\gd(l)=0$ if $l(n)=2n$ and $\gd(l)=1$ if $l(n)=2n\pm 1$.

\begin{enumerate}[label=\upshape{(\alph*)},leftmargin=1cm]
 \item \label{enu:thm-Mixed-Parity-2n2n+1}
If $l_1(n)\ne 2n-1$ then we have
\begin{align}\label{equ-thm-Mixed-Parity}
\Xi(\bfl;\bfs):=\sum_{n_1 \ \underset{1}{\succ}\, \cdots\underset{d-1}{\succ}  n_d\  \underset{d}{\succ} \  0}
 \frac{a_{n_1}}{l_1(n_1)^{s_1}\cdots l_d(n_d)^{s_d}} \in i^{\gd(l_1)}\CMZV_{|\bfs|}^4 ,
\end{align}
where $\underset{j}{\succ}$ is ``$\ge$'' if $l_j(n)=2n+1$ and is $``>''$ otherwise.

 \item \label{enu:thm-Mixed-Parity-gen}
Suppose $s_1\ge 2$ if $l_1(n)=2n-1$ in which case we set $\nu(l_1)=1$, and set $\nu(l_1)=0$ otherwise.
Then we have
\begin{align}\label{equ-thm-Mixed-Parity1}
\sum_{n_1 \ \underset{1}{\succ} \, \cdots\underset{d-1}{\succ} n_d\ \underset{d}{\succ} \ 0}
 \frac{a_{n_1}}{l_1(n_1)^{s_1}\cdots l_d(n_d)^{s_d}} \in    \CMZV_{|\bfs|-\nu(l_1)}^4 \otimes \Q[i].
\end{align}

 \item \label{enu:thm-Mixed-Parity-Any}
 Moreover, the claim in \eqref{equ-thm-Mixed-Parity1} still holds if one changes any of the strict inequalities
$n_j>n_{j+1}$ to $n_j\ge n_{j+1}$ and vice versa, , provided the series is defined.
Here we set $n_{d+1}=0$. In particular,
\begin{align}\label{equ:thm-Mixed-Parity-Any}
\sum_{n_1 \succ n_2 \succ\, \cdots \succ n_d\succ\,0}
 \frac{a_{n_1}}{l_1(n_1)^{s_1}\cdots l_d(n_d)^{s_d}} \in  \CMZV_{\le |\bfs|-\nu(l_1)}^4\otimes\Q[i],
\end{align}
where ``$\succ$'' can be either ``$\ge$'' or ``$>$'', provided the series is defined.
\end{enumerate}
\end{thm}

\begin{proof} We first claim that we may reduce the sums in the first two cases to
those sums where $l_j(n)=2n-1$ appears only when $j=1$, if it ever appears.
We can prove this by induction on the depth in exactly the same way as was used in
the proof of \cite[Thm.~4.2(b)]{XuZhao2022a}. So leave this to the interested reader.

\medskip
\ref{enu:thm-Mixed-Parity-2n2n+1}
Set $F_s=(\cot t\, dt)^{s-1}$.
Recall that \eqref{equ-ga-depthd} and \eqref{equ-gbE-depthd} provide us the following iterative structure:
\begin{align}\label{equ-ga-depthd2}
\sum_{m>n} \frac{a_m(\sin y)}{(2m)^s}
 =&\,\int_0^y F_s (1-\csc t \,dt\circ\sec t) a_n(\sin t) \tan t\,dt,\\
\sum_{m\ge n} \frac{a_m(\sin y) }{(2m+1)^s}
 =&\,\csc y\int_0^y F_s (\csc t-dt\circ\sec t) a_n(\sin t) \tan t\,dt. \label{equ-gbE-depthd2}
\end{align}
The key idea is to use \eqref{equ-ga-depthd2} and \eqref{equ-gbE-depthd2} repeatedly to express \eqref{equ-thm-Mixed-Parity}
as an iterated integral and then use the change of variables $t\to \sin^{-1} [(1-t^2)/(1+t^2)]$ to convert this iterated integral
to a $\Q$-linear combination of iterated integrals that are clearly in $\CMZV_{|\bfs|}^4.$

To begin, similarly to the proof of Thm.~\ref{thm-Mixed-Parity}
we call the iterated integral in an iteration of \eqref{equ-ga-depthd2} a $\ga$-block and
the iterated integral in an iteration of \eqref{equ-gbE-depthd2} a $\gb^\star$-block.
The extra function $\csc y$ before the $\gb^\star$-block brings the main complication into this process since it changes the shape of
the block in front of it.
Then we have four different cases for the first block:

\begin{enumerate}[label=(\bf{1st-block},leftmargin=3cm]
 \item $\boldsymbol{\ga}\circ\ga$) $\ga$-block, followed by $\ga$-block:
\begin{align}\label{1stgaga-block}
\hskip-2.5cm F_s (1-\csc t \,dt\circ\sec t) \tan t\,dt
=&\, F_s(\tan t\,dt-\csc t\sec t \,dt+\csc t \,dt \circ\sec t ) \nonumber \\
=&\, F_s(\csc t \,dt \circ\sec t-\cot t \,dt).
\end{align}

 \item $\boldsymbol{\ga}\circ\gb^\star$) $\ga$-block, followed by a $\gb^\star$-block:
\begin{align}\label{1stgagb-block}
\hskip-2.5cm F_s(1-\csc t \,dt\circ\sec t) \sec t\,dt
=&\, F_s(\sec t\,dt- \csc t dt \, d\tan t)\nonumber \\
=&\, F_s( \csc t \, dt\circ\tan t).
\end{align}

 \item $\boldsymbol{\gb^\star}\circ\ga$) $\gb^\star$-block, followed by a $\ga$-block:
\begin{align}\label{1stgbga-block}
\hskip-2.5cm F_s(\csc t-dt\circ\sec t) \tan t\,dt
=F_s(\sec t\,dt - dt\, d\sec t)=F_s( dt\circ\sec t).
\end{align}

 \item $\boldsymbol{\gb^\star}\circ\gb^\star$) $\gb^\star$-block, followed by another $\gb^\star$-block:
\begin{align}\label{1stgbgb-block}
\hskip-2.5cm F_s(\csc t-dt\circ\sec t) \sec t\,dt
=F_s(\csc t \sec t\, dt - dt\, d\tan t)
=F_s(\cot t \, dt + dt\circ\tan t).
\end{align}
\end{enumerate}
The 1-forms appearing in the above are listed as follows:
\begin{align}\label{starting1-forms}
 \csc t \,dt , \quad\cot t\,dt, \quad dt.
\end{align}
Moreover, the following observation is crucial later:
\begin{align}\label{dtObserve}\tag{$\star$}
\aligned
\  & \text{The 1-form $dt$ appears only when the first block is $\gb$ and it always} \\
\  & \text{has a trailing $\sec t$ (resp. $\tan t$) if the next block is $\ga$ (resp. $\gb$).}
\endaligned
\end{align}
Hence, all the blocks after the first may (or may not) be multiplied by either $\tan t$ or $\sec t$. Thus
for the middle blocks (i.e., neither initial nor end) we have the following cases:
\begin{alignat}{4}
&\, \text{\bf Mid-block } \boldsymbol{\ga}\circ\ga: \qquad &\, & (\text{$1$ or $\sec t$})F_s(\csc t \,dt \circ\sec t-\cot t \,dt), \label{midgaga-block}\\
&\, \text{\bf Mid-block } \boldsymbol{\ga}\circ\gb^\star: \qquad &\, & (\text{$1$ or $\sec t$})F_s(\csc t\, dt\circ\tan t), \label{midgagb-block}\\
&\, \text{\bf Mid-block } \boldsymbol{\gb^\star}\circ\ga: \qquad &\, & (\text{$1$ or $\tan t$})F_s(dt\circ\sec t), \label{midgbga-block} \\
&\, \text{\bf Mid-block } \boldsymbol{\gb^\star}\circ\gb^\star: \qquad &\, & (\text{$1$ or $\tan t$})F_s(\cot t \, dt + dt\circ\tan t). \label{midgbgb-block}
\end{alignat}
Therefore the following additional 1-forms may appear:
\begin{align}\label{mid1-forms}
 \tan t\,dt, \quad \sec t\csc t \,dt.
\end{align}
We now turn to the ending block. Since
\begin{align*}
(\tan t\,dt-\csc t \,dt \int_0^t \sec x\tan x\, dx)&\, = (\tan t\,dt-(\sec t-1)\csc t \,dt)=(\csc t-\cot t)\,dt,\\
(\tan t\csc t\,dt -dt\int_0^t \sec x\tan x\, dx)&\, =(\sec t\,dt -(\sec t-1) \,dt)=dt,
\end{align*}
we may have the following forms for the end block:
\begin{alignat}{4}
&\, s_d\ge 2,\ \text{\bf End-block } \boldsymbol{\ga}: \qquad &\, & (\text{$1$ or $\sec t$})F_s (\csc t-\cot t)\,dt, \label{endga-block}\\
&\, s_d\ge 2,\ \text{\bf End-block } \boldsymbol{\gb^\star}: \qquad &\, & (\text{$1$ or $\tan t$})F_s \,dt. \label{endgb-block}
\end{alignat}
If $s=1$ then since
\begin{align*}
\tan t(\csc t-\cot t)=\sec t-1, \quad \sec t(\csc t-\cot t)=\csc t(\sec t-1)
\end{align*}
we see that the end block has the form
\begin{alignat}{4}
&\, s_d=1,\ \text{\bf End-block } \boldsymbol{\ga}: \qquad &\, &
(\csc t-\cot t)\, dt \text{ or } \csc t(\sec t-1)\,dt, \label{endga-block-Depth1} \\
&\, s_d=1,\ \text{\bf End-block } \boldsymbol{\gb^\star}: \qquad &\, & dt \text{ or } \tan t\,dt. \label{endgb-block-Depth1}
\end{alignat}
Under the change of variables $t\to \sin^{-1} [(1-t^2)/(1+t^2)]$ we have
\begin{align}\label{equ-1forms-ta-tx}
dt \to i \td_{-i,i}, \
\cot t\,dt \to \ty,\
\tan t\,dt \to \tz,\
\csc t \,dt \to \td_{-1,1}, \
\sec t \,dt \to -\ta, \
\sec t\csc t \,dt \to \ty+\tz,
\end{align}
where $\td_{\xi,\xi'}=\tx_\xi-\tx_{\xi'}$, $\ty=\tx_{-i}+\tx_{i}-\tx_{-1}-\tx_{1}$ and $\tz=-\ta-\tx_{-i}-\tx_{i}$.
We see that under the above change of variables, which reverses the order of the 1-forms,
the 1-form $\ta$ does not appear at the end (see \eqref{starting1-forms}).
On the other hand, the only 1-forms that can appear at the beginning are:
\begin{align*}
dt \to i \td_{-i,i}, \quad
\tan t\,dt \to \tz,\quad
(\csc t-\cot t)\, dt \to \td_{-1,1}-\ty=2\tx_{-1}-\tx_{-i}-\tx_{i},\\
\sec t \,dt \to -\ta, \quad
\csc t(\sec t-1)\,dt \to \ty+\tz-\td_{-1,1}=-\ta-2\tx_{-1}.
\end{align*}
The key observation is that $\tx_1$ does not appear at the beginning.
Consequently, all the iterated integrals are convergent and provide
the real or the imaginary part of some admissible CMZVs of level 4.

To determine exactly whether it is the real or the imaginary part, we need to count
the number of $i \td_{-i,i}$'s, the only 1-form that has the imaginary coefficient $i$,
which is produced only by the original 1-form $dt$ in the trigonometric iterated integral expression.
To do this, we break into two cases, guided by the crucial observation \eqref{dtObserve} above.
To save space, we denote by $N(dt)$ the number of 1-form $dt$ in the trigonometric iterated integral expression
of the sum in \eqref{equ-thm-Mixed-Parity}.

(A) The starting block is $\gb$. When followed by another $\gb^\star$-block the trailing $\tan t$ can only be combined with two other forms,
namely, $\cot t\,dt$ or $dt\circ\tan t$ (see \eqref{midgbgb-block}),
which produces $dt^2$ or $dt (\tan t\,dt) \circ\tan t$. Repeating this until the end block
if no $\ga$-block appears, or until a transition $\gb^\star$-$\ga$ block chain emerges \eqref{midgbga-block},
we see that either (i) there are even number of $dt$'s without trailing $\circ\tan t$ or (ii) there are
odd number of $dt$'s with a trailing $\circ\tan t$. If $\ga$-block does not appear at all then the end block is given
by \eqref{endgb-block} ($s=s_d\ge 2$) and \eqref{endgb-block-Depth1} ($s=s_d=1$) so that the parity changes in case (i) or
the parity doesn't change in case (ii). If there is a transition $\gb^\star$-$\ga$ block chain, then there is a trailing $\circ\sec t$ produced
while there are odd number of $dt$'s in front of $\circ\sec t$.
By case (B) below we see that the number of $dt$'s after this transition must be even.
To summarize, we see that $N(dt)$ is always odd if the starting block is $\gb$.

(B) The starting block is $\ga$. If there is no $\gb^\star$-block then clearly $N(dt)=0$ by \eqref{1stgaga-block},
\eqref{midgaga-block}, \eqref{endga-block} and \eqref{endga-block-Depth1}. Suppose $\gb^\star$-block does appear.
By \eqref{midgbga-block} and \eqref{midgbgb-block}, we see that this block produces either
\begin{enumerate}[label=(\roman*),leftmargin=2cm]
  \item  even number of $dt$'s with a trailing $\sec t$ when followed by a $\ga$-block, or
  \item  even number of $dt$'s with a trailing $\tan t$ when followed by a $\gb^\star$-block, or
  \item odd number of $dt$'s without a trailing function.
\end{enumerate}
The case (i) leads to no $dt$'s until the next $\gb^\star$-block appears.
For the other two cases, we can argue exactly as in case (A) above and
show that these repeat until the end or a transition $\ga$-$\gb$ block chain (back to case (i)).
Repeating the above argument in the three cases (i)--(iii) we see that until the end we still have the same three cases.
If the end block is $\ga$ then we must be back in case (i) and by \eqref{endga-block} and \eqref{endga-block-Depth1}
this block does not produce the 1-form $dt$ so that $N(dt)$ is even.
If the end the block is $\gb$ then it has the form $\tan t F_s\,dt$ in case (ii) and $F_s\,dt$ in case (iii).
In case (ii) either two $dt$'s or no $dt$ is produced so it doesn't change
of the parity of $N(dt)$. In case (iii) one $dt$ is produced which changes $N(dt)$ to even. To summarize, we see
that $N(dt)$ is always even if the starting block is $\ga$.

\medskip
\ref{enu:thm-Mixed-Parity-gen} and
\ref{enu:thm-Mixed-Parity-Any} can be proved by the same proof as that for \cite[Thm.~4.2(b)(c)(d)]{XuZhao2022a}.
We may first reduce the general case to the case where if $\gam$-block appears then it only appears as the first block,
in which case we can assume the weight of this block $s\ge2$ by Thm.~\ref{thm-LeadGammaBlock}. We now further
consider two subcases: (i) second block is a $\ga$-block of weight $b$ and (ii) second block is a $\gb^\star$-block of weight $b$.
Thus we have the following two kinds of iterated integrals to consider:
\begin{align*}
(i):&\, \qquad  \int_0^{\pi/2} F_{s-1} \frac{dt}{\tan^2 t} \tan t \sec t \, dt F_{b-1} (\csc t-\cot t)\,dt \cdots ,\\
(ii):&\, \qquad  \int_0^{\pi/2} F_{s-1} \frac{dt}{\tan^2 t} \sec^2 t \, dt F_{b-1}  \,dt \cdots.
\end{align*}
The claim in the theorem then follows immediately from the computation in Examples \ref{eg-gam-ga-block} and
\ref{eg-gam-gb-block}. Indeed, from the examples we see that the claims in the theorem hold when
the first two blocks are given by either $\gam$-$\ga$-block or $\gam$-$\gb^\star$-block chains.
But if there are more blocks after these two, the argument in case (a) applies since these additional
one are either $\ga$-blocks or $\gb^\star$-blocks.

This completes the proof of the theorem.
\end{proof}

\begin{re}
The theorem generalizes the first inclusion relation in \cite[Thm.~9.6]{XuZhao2021b}.
\end{re}

\begin{cor}\label{cor-Mixed-Parity-zeta-gbE}
Suppose $d\in \N$, $\bfs=(s_1,\dots,s_d)\in\N^d$. Let $l_1(n),\dots,l_d(n)=2n,2n+1$.
Then we have
\begin{align}
\sum_{n_1>n_2\succ \dots\succ n_d\  \underset{d}{\succ} \ 0}
 \frac{a_{n_1}}{n_1^{s_1}l_2(n_2)^{s_2}\cdots l_d(n_d)^{s_d}} \in &\, \CMZV_{|\bfs|}^4 ,  \label{equ-cor-Mixed-Parity1}\\
\sum_{n_1\ge n_2\succ \dots\succ n_d\  \underset{d}{\succ} \ 0}
 \frac{a_{n_1}}{(2n_1+1)^{s_1}l_2(n_2)^{s_2}\cdots l_d(n_d)^{s_d}} \in &\, i\CMZV_{|\bfs|}^4 ,  \label{equ-cor-Mixed-Parity2}\\
\sum_{n_1\succ \dots\succ n_d \  \underset{d}{\succ} \ 0}
 \frac{a_{n_1}}{l_1(n_1)^{s_1}\cdots l_d(n_d)^{s_d}} \in &\,  \CMZV_{|\bfs|}^4 \otimes \Q[i], \label{equ-cor-Mixed-Parity}
\end{align}
where $\succ$ is either ``$\ge$'' or $``>''$.
\end{cor}
\begin{proof}
The proof easily follows from the fact that by using the Principle of Inclusion and Exclusion we may
convert $\ge$ to $>$ (and vice versa) by using partial fractions.
\end{proof}

\begin{re}
Since the proofs of Thm.~\ref{thm-Mixed-Parity-all} and Cor.~\ref{cor-Mixed-Parity-zeta-gbE}
are both constructive we see that every sum of the form
\eqref{equ-thm-Mixed-Parity}, or more generally of the form \eqref{equ-cor-Mixed-Parity},
can be computed exactly in terms of CMZVs of level 4.
\end{re}

\section{Ap\'ery-type series involving squares of central binomial coefficients}
In this section, we will consider another class of Ap\'ery-type series by replacing $a_n$ by $a_n^2$
in all the series appeared in the previous sections.

\begin{thm}\label{thm-binnSquare}
Keep notation as in Thm.~\ref{thm-Mixed-Parity-all}. Assume $s_1\ge 3$.

\begin{enumerate}[label=\upshape{(\alph*)},leftmargin=1cm]
 \item  \label{enu:thm-binnSquare-2n2n+1}
 Let $l_1(n),\dots,l_d(n)=2n,2n+1$. 
Then we have
\begin{align}\label{equ-sqare-Mixed-Parity1}
\sum_{n_1 \ \underset{1}{\succ} \ \cdots\underset{d-1}{\succ} n_d\ \underset{d}{\succ} \, 0}
 \frac{ a_{n_1}^2}{l_1(n_1)^{s_1}\cdots l_d(n_d)^{s_d}} \in  \frac{i}{\pi}\CMZV_{|\bfs|+1}^4.
\end{align}

 \item \label{enu:thm-binnSquare-gen}
 More generally, if $l_1(n),\dots,l_d(n)=2n,2n\pm 1$ then we have
\begin{align}\label{equ-sqare-Mixed-Parity2}
\sum_{n_1 \ \underset{1}{\succ} \ \cdots\underset{d-1}{\succ} n_d\ \underset{d}{\succ} \, 0}
 \frac{ a_{n_1}^2}{l_1(n_1)^{s_1}\cdots l_d(n_d)^{s_d}} \in  \frac{1}{\pi}\CMZV_{\le \max\{|\bfs|+1-\eta(l_1),\iota(l_2)\}}^4\otimes\Q[i],
\end{align}
where $\eta(l)=2$ if $l(n)=2n-1$ and $\eta(l)=0$ otherwise,
 $\iota(l)=2$ if $l(n)=2n$ and $\iota(l)=1$ otherwise (including $\iota(\emptyset)=1$).

 \item  \label{enu:thm-binnSquare-Strict}
 Moreover,
\begin{align*}
\sum_{n_1\underset{1}{\succ} n_2 \succ\cdots \succ n_d\succ \, 0}
 \frac{ a_{n_1}^2}{l_1(n_1)^{s_1}\cdots l_d(n_d)^{s_d}} \in \frac{1}{\pi}\CMZV_{\le \max\{|\bfs|+1-\eta(l_1),\iota(l_2)\}}^4\otimes\Q[i],
\end{align*}
where  ``$\succ$'' is either ``$\ge$'' or ``$>$'', provided the series is defined.
\end{enumerate}
\end{thm}
\begin{proof} The key observation is to apply the Wallis integrals
\begin{align*}
 \int_0^1 \frac{x^{2n}}{\sqrt{1-x^2}} \,dx=\int_0^{\pi/2} \sin^{2n} t \,dt= \frac12 B \Big(n+\frac12,\frac12\Big)=\frac{\pi}{2}a_{n}.
\end{align*}
In the following proof, we drop the restriction on the summation indices to save space.

\medskip
(i) When $l_1(n)=2n+1$ by \eqref{equ-Mixed-Parity-Trig} we see that the sum
\begin{align*}
 \sum
\frac{a_{n_1}(\sin y) }{ l_1(n_1)^{s_1}\cdots l_d(n_d)^{s_d}}
 =\csc y \int_0^y F_{s_1}(\sin t\tan t\,dt-dt\, d(\sec t) ) \circ  \prod_{j=2}^{d}{}\hskip-2.35ex{\scriptstyle \circ}\hskip1ex
 \gl_{l_j,s_j}  \circ \ola{a_n(\sin t)}.
\end{align*}
Thus integrating over $(0,\pi/2)$ and dividing by $\pi/2$ we get
\begin{align}\label{equ-binnSquare-gb}
\sum
 \frac{a_{n_1}^2}{l_1(n_1)^{s_1}\cdots l_d(n_d)^{s_d}}
 =  \frac{2}{\pi} \int_0^{\pi/2}\csc t\,dt\, F_{s_1}(\sin t\tan t\,dt-dt\, d(\sec t) ) \prod_{j=2}^{d}{}\hskip-2.35ex{\scriptstyle \circ}\hskip1ex
 \gl_{l_j,s_j}  \circ \ola{a_n(\sin t)}.
\end{align}
The claim follows immediately since there are odd number of $dt$'s in this case, as shown in
the proof of Thm.~\ref{thm-Mixed-Parity-all}\ref{enu:thm-Mixed-Parity-2n2n+1}.

\medskip
(ii) If $l_1(n)=2n$ then we see that
\begin{align*}
 \sum
 \frac{a_{n_1}(\sin y)}{l_1(n_1)^{s_1}\cdots l_d(n_d)^{s_d}}
 = \int_0^y F_{s_1} (\tan t\,dt-\csc t \,dt\, d(\sec t) )\circ  \prod_{j=2}^{d}{}\hskip-2.35ex{\scriptstyle \circ}\hskip1ex
 \gl_{l_j,s_j} \circ \ola{a_n(\sin t)}.
\end{align*}
Thus integrating over $(0,\pi/2)$ and dividing by $\pi/2$ we get
\begin{align}\label{equ-binnSquare-ga}
 \sum
 \frac{a_{n_1}^2}{l_1(n_1)^{s_1}\cdots l_d(n_d)^{s_d}}
 =\frac{2}{\pi} \int_0^{\pi/2} dt\, F_{s_1} (\tan t\,dt-\csc t \,dt\, d(\sec t) )\circ \prod_{j=2}^{d}{}\hskip-2.35ex{\scriptstyle \circ}\hskip1ex
 \gl_{l_j,s_j} \circ \ola{a_n(\sin t)}.
\end{align}
The corollary holds as well in this case as the number of $dt$'s is changed to odd because of the leading $dt$,
since originally, as shown in
the proof of Thm.~\ref{thm-Mixed-Parity-all}\ref{enu:thm-Mixed-Parity-2n2n+1}, the number of $dt$'s was even.
The appearance of $\iota(l)$ is due to the special behavior of $\ga$-block
as manifested by Example~\ref{eg-binn-sqr-gam-ga-block} and Example~\ref{eg-binn-sqr-gam-ga-block2} in Appendix B.

\medskip
(iii)  If $l_1(n)=2n-1$ then there are two cases. By \eqref{equ-Mixed-Parity-Trig} if $s_1=1$,
then we have
\begin{align*}
 \sum
 \frac{a_{n_1}(\sin y)}{l_1(n_1)^{s_1}\cdots l_d(n_d)^{s_d}}
 = \cos y\int_0^y d(\sec t)\circ\prod_{j=2}^{d}{}\hskip-2.35ex{\scriptstyle \circ}\hskip1ex
 \gl_{l_j,s_j}  \circ \ola{a_n(\sin t)}.
\end{align*}
Thus integrating over $(0,\pi/2)$ and dividing by $\pi/2$ we get
\begin{align*}
 \sum
 \frac{a_{n_1}^2}{l_1(n_1)^{s_1}\cdots l_d(n_d)^{s_d}}
=&\, \frac{2}{\pi} \int_0^{\pi/2} \cos t\,dt\, d(\sec t)\circ \prod_{j=2}^{d}{}\hskip-2.35ex{\scriptstyle \circ}\hskip1ex
 \gl_{l_j,s_j}  \circ \ola{a_n(\sin t)}.
\end{align*}
If $s_1\ge2$, then
\begin{align}\label{equ-binnSquare-gam1}
\sum
 \frac{a_{n_1}(\sin y)}{l_1(n_1)^{s_1}\cdots l_d(n_d)^{s_d}}
 = \sin y\int_0^y F_{s-1} \frac{dt}{\tan^2 t} d(\sec t)\circ \prod_{j=2}^{d}{}\hskip-2.35ex{\scriptstyle \circ}\hskip1ex
 \gl_{l_j,s_j} \circ \ola{a_n(\sin t)}.
\end{align}
Thus integrating over $(0,\pi/2)$ and dividing by $\pi/2$ we get
\begin{align}\label{equ-binnSquare-gam-depthd}
\sum
 \frac{a_{n_1}^2}{l_1(n_1)^{s_1}\cdots l_d(n_d)^{s_d}}
=\frac{2}{\pi} \int_0^{\pi/2} \sin t\,dt\, F_{s-1} \frac{dt}{\tan^2 t} d(\sec t)\circ \prod_{j=2}^{d}{}\hskip-2.35ex{\scriptstyle \circ}\hskip1ex
 \gl_{l_j,s_j}  \circ \ola{a_n(\sin t)}.
\end{align}

To complete the proof of the theorem we only need to pay attention to the weight increasing in (a) and weight
drop phenomenon in (b) and (c) associated with
a leading $\gam$-block. The first phenomenon in (a) is obvious from (i) and (ii) above. By (iii) it is also
easy to see the weight cannot increase. To show that in fact the may drop by 1
one can carry out a case by case study using the examples given in the next section
and the Appendix. We leave the detail to the interested reader.
\end{proof}

\section{A corollary and some examples}
In this last section, we first answer affirmatively a few questions we posted at the end of \cite{XuZhao2021b}.
For $\bfk\in\N^d$ and $\bfl\in\N^e$ we define
\begin{align*}
\ze_n(\bfk):=&\, \sum_{n\ge m_1>\dots>m_d>0} \frac{1}{m_1^{k_1}\cdots  m_d^{k_d}}, \\
t_n(\bfl):=&\,  \sum_{n\ge  r_1>\dots>r_e>0} \frac{1}{(2r_1-1)^{l_1}\cdots (2r_e-1)^{l_e}}.
\end{align*}

\begin{cor}\label{cor-answerQuestions}
For all $m\in\N$, $p\in\N_{\ge2}$, $q\in\N_{\ge3}$,
and all compositions of positive integers $\bfk$ and $\bfl$ (including the cases $\bfk=\emptyset$ or $\bfl=\emptyset$),
we have
\begin{alignat*}{4}
{\rm (a)}& \ \ \sum_{n=1}^\infty a_n \frac{\ze_n(\bfk)t_n(\bfl)}{n^{p}}  \in \CMZV_{|\bfk|+|\bfl|+p}^4,& \quad
{\rm (b)}& \ \ \sum_{n=1}^\infty a_n^2 \frac{\ze_n(\bfk)t_n(\bfl)}{n^{q}}  \in \frac{i}{\pi}\CMZV^{4}_{|\bfk|+|\bfl|+q+1},\\
{\rm (c)}& \ \ \sum_{n=0}^\infty a_n \frac{\ze_n(\bfk)t_n(\bfl)}{(2n+1)^{p}}  \in  i\CMZV^{4}_{|\bfk|+|\bfl|+p},& \quad
{\rm (d)}& \ \ \sum_{n=0}^\infty a_n^2 \frac{\ze_n(\bfk)t_n(\bfl)}{(2n+1)^{q}} \in \frac{i}{\pi} \CMZV^{4}_{|\bfk|+|\bfl|+q+1}.
\end{alignat*}
\end{cor}
\begin{proof} Write
\begin{align*}
\ze_n(\bfk)=\sum_{n\ge m_1>\dots>m_d>0} \frac{1}{m_1^{k_1}\cdots  m_d^{k_d}},
\quad
t_n(\bfl)= \sum_{n> r_1>\dots>r_e \ge 0} \frac{1}{(2r_1+1)^{l_1}\cdots (2r_e+1)^{l_e}}.
\end{align*}
We only need to note
the following facts: (i) for any summation index $m$ for $\ze_n(\bfk)$ and summation index $r$ for $t_n(\bfl)$ there
are only two possibilities: $m>r$ or $r\ge m$; (ii) we can re-write
\begin{align*}
    \sum_{n>r_1} \frac{1}{(2n+1)^q (2r_1+1)^{l_1}}= \sum_{n\ge r_1} \frac{1}{(2n+1)^q (2r_1+1)^{l_1}}-\frac{1}{(2n+1)^{q+l_1}}
\end{align*}
and obtain similar identities when $n$ and $r_1$ are replaced by $r_{j}$ and $r_{j+1}$. Therefore, we see that
(a) and (c) are special cases of Thm.~\ref{thm-Mixed-Parity-all}\ref{enu:thm-Mixed-Parity-2n2n+1}.
(b) and (d) are special cases of Thm.~\ref{thm-binnSquare}\ref{enu:thm-binnSquare-2n2n+1}.
\end{proof}

In the following examples we first convert the Ap\'ery-type series to CMZVs of level 4
by Thm~\ref{thm-Mixed-Parity-all} and
then use Au's Mathematica package \cite{Au2020} to simplify the expressions.

\begin{ex}\label{eg-gam-ga-block}
We now consider series given by a $\gam$-$\ga$-block chain. By Thm.~\ref{thm-LeadGammaBlock} we have
\begin{align*}
\Xi(2n-1,2n;1,1)=\Xi(2n;1)=\log 2 \quad\text{and} \quad \Xi(2n-1,2n;1,b)=\Xi(2n;b), \ \forall b\in\N.
\end{align*}
Next,  for any $s\ge 2$ by \eqref{equ-gam-depthd} and \eqref{endga-block-Depth1} we get
\begin{align*}
\Xi(2n-1,2n;s,1)=&\,\sum_{n_1>n_2>0} \frac{a_{n_1}}{(2n_1-1)^s(2n_2)}
=\int_0^{\pi/2} F_{s-1} \frac{dt}{\tan^2 t} \tan t \sec t \, dt (\csc t-\cot t)\,dt\\
=&\,(-1)^s \int_0^1 (2\tx_{-1}-\tx_i-\tx_{-i})- (-1)^s \sum_{j=0}^{s-2} i\int_0^1 (\ta+2\tx_{-1})\td_{-i,i}\ty^j,
\end{align*}
after applying suitable change of variables, see \eqref{equ-1forms-ta-tx}.
Note the highest CMZV weight drops by 1 as predicted by Thm.~\ref{thm-Mixed-Parity-all}\ref{enu:thm-Mixed-Parity-gen}.
For example,
\begin{align*}
&\,\Xi(2n-1,2n;2,1)=  2 G-\frac12\pi\log 2-\log 2 \approx 0.04999096264,\\
&\,\Xi(2n-1,2n;3,1)= \frac{3\pi^3}{32} -4\Im\Li_3\Big(\frac{1+i}2\Big) -\frac{\pi\log 2}{8} \Big(3\log2 -4\Big)+\log 2-2G\approx 0.010517475685,
\end{align*}
where $G=\sum_{k\ge 0}\frac{(-1)^k}{(2k+1)^2}$ is Catalan's constant.
Next, for all $s,b\ge 2$ we have
\begin{align*}
\Xi(2n-1,2n;s,b)=&\,\sum_{n_1>n_2>0} \frac{a_{n_1}}{(2n_1-1)^s(2n_2)^b}
=\int_0^{\pi/2} F_{s-1} \frac{dt}{\tan^2 t} \tan t \sec t \, dt \, F_b (\csc t-\cot t)\,dt\\
=&\,-(-1)^{s+b} \int_0^1 (2\tx_{-1}-\tx_i-\tx_{-i}) \bigg(\ty^{b-1}+ i \sum_{j=0}^{s-2}\ty^{b-2}\td_{-1,1} \td_{-i,i} \ty^j \bigg).
\end{align*}
\end{ex}

\begin{ex}\label{eg-gam-gb-block}
We now consider series given by a $\gam$-$\gb^\star$-block chain. By Thm.~\ref{thm-LeadGammaBlock} we have
\begin{align*}
\Xi(2n-1,2n+1;1,1)= \sum_{n_1>n_2\ge 0} \frac{a_{n_1}}{(2n_1-1)(2n_2+1)} = \Xi(2n+1;1)= \frac{\pi}{2}.
\end{align*}
For any $s\ge 2$ by \eqref{equ-gam-depthd} and \eqref{endgb-block-Depth1} we get (see Example \ref{eg-gam-gb-block-App})
\begin{align*}
\Xi(2n-1,2n+1;s,1)=&\, i\int_0^1 \sum_{j=2}^s (-1)^{s-j}\Big(F_j \,dt + F_{j-1} \,dt \tan t \,dt\Big)-(-1)^{s}\frac{\pi}{2}\\
=&\,  (-1)^s \bigg(i\int_0^1\sum_{j=2}^s\Big(\td_{-i,i} \ty^{j-1}+\tz\td_{-i,i}\ty^{j-2}\Big)-\frac{\pi}{2}\bigg).
\end{align*}
Hence the highest CMZV weight again drops by 1 as predicted by Thm.~\ref{thm-Mixed-Parity-all}\ref{enu:thm-Mixed-Parity-gen}.
Further, for all $s,b\ge 2$ we see that
\begin{align*}
\Xi(2n-1,2n+1;s,b)=&\,\sum_{n_1>n_2\ge 0} \frac{a_{n_1}}{(2n_1-1)^s(2n_2+1)^b}\\
=&\,-i(-1)^{s+b}\int_0^1 \td_{-i,i}\ty^{b-2}  \ty+i (-1)^{s+b}  \sum_{j=0}^{s-2} \int_0^1
    \td_{-i,i} \Big(\ty^{b-2} \td_{-i,i}^2 \ty^j- \ty^{j+b} \Big). 
\end{align*}
\end{ex}

Next, we present two examples illustrating the ideas in Thm.~\ref{thm-binnSquare}.

\begin{ex} \label{gam-depthd}
We consider series given by a $\gam^{\circ d}$-block chain with leading weight two and trailing weight one blocks.
By  \eqref{equ-gam-d=1} and the proof of Thm.~\ref{thm-binnSquare} (see Example \ref{gam-depthd-App})
\begin{align*}
S_d:=&\, \sum_{n_1>n_2>\cdots>n_d> 0} \frac{a_{n_1}^2}{(2n_1-1)^2(2n_2-1)\cdots(2n_d-1)}\\
=&\, \frac{2}{\pi} \bigg(d+1-\frac{\pi}{2}
+\sum_{j=0}^{d-2} (-1)^{j}(d-j) \int_0^1  \tz^j (\tx_{-i}+\tx_{i}) \bigg).
\end{align*}
This shows the weight may drop by two in this special case.
\end{ex}

\begin{ex} \label{eg-binn-sqr-gam-ga-block}
We consider series given by a $\gam$-$\ga$-block chain with $\ga$-block having weight one.
When $\gam$-block has weight one there is no weight drop from the last example.
But with higher weight $\gam$-block, the weight drop pattern resumes:
\begin{align*}
\sum_{n_1>n_2> 0} \frac{a_{n_1}^2}{(2n_1-1)^2(2n_2)} =\frac{2}{\pi}\Big(2G-\pi\log2 +\pi - 4\log 2\Big) \approx 0.01486445.
\end{align*}
\end{ex}

We have given more details of the above examples and computed many more examples
in the Appendix. The interested reader may check
and see more subtle patterns contained in these examples.

We now turn to some identities we found in the literature.
For $n,k\in\N$, it is conventional to define the harmonic numbers and generalized harmonic numbers by
\begin{align*}
  H_n:=\ze_n(1)=1+\frac12+\frac13+\cdots+\frac1n \quad\text{and} \quad H_n^{(k)}:=\ze_n(k)=1+\frac1{2^k}+\frac13+\cdots+\frac1{n^k}
\end{align*}
respectively. From the examples above and those contained in the Appendix, we can derive immediately the following identity
which also appeared in \cite[Thm.~2.7]{Campbell2019} and \cite{CampbellSofo2017}:
\begin{align*}
 \sum_{n>0} \frac{a_n^2 H_{2n}}{(2n-1)^2}
=&\,\sum_{n>0} \frac{a_n^2}{(2n-1)^3}+\sum_{n>m>0} \frac{a_n^2}{(2n-1)^2(2m-1)}+\sum_{n>m>0} \frac{a_n^2}{(2n-1)^2(2m)}\\
&\, +\sum_{n>0} \frac{a_n^2}{2n-1)^2}-\sum_{n>0} \frac{a_n^2}{(2n-1)}+\sum_{n>0} \frac{a_n^2}{(2n)}\\
=&\, \frac{2}{\pi}\bigg(\Big(\frac{\pi}{2}+2G-3\Big)  + \Big( 3-\frac{\pi}2-2\log2\Big)   +   \Big(2G-\pi\log2 +\pi-4\log 2\Big)   \\
&\,  +     \Big( 2-\frac{\pi}2\Big)   -  \Big(\frac{\pi}{2}-1\Big) +   \Big( \pi\log2-2G \Big)\bigg)\\
=&\, \frac{4G-12\log 2+6}{\pi}.
\end{align*}
Similarly, we can also verify:
\begin{alignat*}{4}
 \sum_{n>0} \frac{a_n^2 H_{2n}}{2n-1}
=&\, \frac{2}{\pi}(3\log 2-1)     &&   (\text{\cite[Thm.~2.5]{Campbell2019} and \cite[Thm.~5.15]{CampbellDS2019}}),\\
\sum_{n>0} \frac{a_n^2 H_n}{2n-1}
=&\, \frac{8\log 2-4}{\pi}   &&  (\text{\cite[Thm. 1]{Campbell2018}}),\\
 \sum_{n>0} \frac{a_n^2 H_n}{(2n-1)^2}
=&\, \frac{12-16\log 2}{\pi}  &&  (\text{\cite[Thm. 2]{Campbell2018} and \cite[p.~10]{Chen2016}}),\\
 \sum_{n>0} \frac{a_n^2 (H_n^2+H_n^{(2)})}{2n-1}
=&\, \frac{4\pi}{3}-\frac{32\log^2 2-32\log 2+16}{\pi} \qquad  &&  (\text{\cite[Thm. 4]{Campbell2018} }).
\end{alignat*}
For the last equation, we may use the stuffle relation
\begin{align}\label{equ-Hstuffle}
 H_n^2=2\zeta_n(1,1)+H_n^{(2)}.
\end{align}

As a further application we can derive \cite[Thm.~5.12]{CampbellDS2019} as follows.
Noting that $(2n-1)a_{n-1}=2na_n$ for all $n\in \N$, we get by shifting index $n\to n-1$
\begin{align*}
\sum_{n>0} \frac{a_{n}^2}{(n+1)}H_n^{(2)}
= \sum_{n>k> 0} \frac{4na_{n}^2}{(2n-1)^2 k^2}
=&\,\sum_{n>k> 0} \frac{2a_{n}^2}{(2n-1)k^2}+\sum_{n>k> 0} \frac{2a_{n}^2}{(2n-1)^2 k^2} \\
=&\, 8V_1+8W_2=\frac{2}{\pi}\Big(16G+\frac{\pi^2}{3}-8\pi\log 2\Big),
\end{align*}
where $V_1$ and $W_2$ are given by Examples \ref{eg-V1} and \ref{eg-W1}, respectively.
Similarly, using the stuffle relation \eqref{equ-Hstuffle} and the identity
\begin{align*}
\sum_{n>0} \frac{a_{n}^2}{(n+1)}\zeta_n(1,1)
=\sum_{n>k>m> 0} &\, \frac{4na_{n}^2}{(2n-1)^2 km}
=\sum_{n>k>m> 0} \frac{2a_{n}^2}{(2n-1)km}+\sum_{n>k>m> 0} \frac{2a_{n}^2}{(2n-1)^2 km}\\
=&\, 8Y_1+8Y_2=\frac{2}{\pi}\Big(16\pi\log^2 2-16G- \pi^2+8\pi\log 2 \Big),
\end{align*}
where $Y_1$ and  $Y_2$ are given by Example \ref{eg-Y1},
we can confirm \cite[Thm.~5.12]{CampbellDS2019} immediately.

\section{Conclusion}
In this paper, by using iterated integrals we have demonstrated that the Ap\'ery-type series
\begin{align*}
\sum_{n_1>\cdots > n_d>0} \frac{\binom{2n_1}{n_1}}{4^{n_1}(2n_1)^{s_1}\cdots (2n_d)^{s_d}}
\quad\text{and}\quad
\sum_{n_1>\cdots > n_d>0} \frac{{\binom{2n_1}{n_1}}^2}{4^{2n_1}(2n_1)^{s_1}\cdots (2n_d)^{s_d}}
\end{align*}
can be expressed as $\Q$-linear combinations of the real and/or imaginary parts of the
colored multiple zeta values of level 4, with an extra factor of $1/\pi$ for the squared version.
The same claim still holds if some or all the indices $2n_j$ are replaced by $2n_j\pm1$
and ``$>$'' replaced by ``$\ge$'' as long as the series converge.

From numerical evidence, it seems
that similar results can be obtained if $4^{n_1}$ is replaced by $8^{n_1}$, $12^{n_1}$ or $16^{n_1}$,
however, the level must be increased significantly. Currently, our method can only
show that these series can be expressed in terms of multiple polylogarithms at suitable algebraic points.
In our next paper \cite{XuZhao2022c}, we will study the alternating versions of the series treated
in this paper and those series in which the binomial coefficients appear on the denominators.

\bigskip
\noindent
{\bf Acknowledgement.} C. Xu is supported by the National Natural Science Foundation of China 12101008, the Natural Science Foundation of Anhui Province 2108085QA01, and the University Natural Science Research Project of Anhui Province KJ2020A0057. J. Zhao is supported by the Jacobs Prize from The Bishop's School.

\appendix


\bigskip

\begin{center}
\huge \textbf{Appendix. Further Examples}
\end{center}

\bigskip
In this appendix,
we provide more details for our computation of the examples in the main text and present many more other examples.

\section{Examples of Ap\'ery-type series involving central binomial coefficients}
In the following examples we first convert the Ap\'ery-type series to CMZVs of level 4
by Thm~\ref{thm-Mixed-Parity-all}
and then use Au's package to simplify the expressions. We also note
that under the change of variables
\begin{align}\label{equ-changVarSeq}
t\to \sin^{-1} t \quad\text{then}\quad t\to \frac{1-t^2}{1+t^2}
\end{align}
we have
\begin{alignat}{6}
\cot t\,dt \to &\, \om_{0}:=\frac{dt}{t}\to \ty,\quad &
 \csc t \,dt \to &\, \om_3:= \frac{dt}{t\sqrt{1-t^2}}\to \td_{-1,1}, \label{1-formChangeVar1} \\
dt \to &\, \om_1:=\frac{dt}{\sqrt{1-t^2}} \to i \td_{-i,i}, \quad &
\sec t\csc t \, dt \to &\, \om_{20}:= \frac{dt}{t(1-t^2)}\to \ty+\tz, \label{1-formChangeVar2} \\
\tan t\,dt \to &\, \om_{2}:= \frac{t\,dt}{1-t^2} \to \tz, \quad &
\sec t\,dt \to &\, \om_8:= \frac{dt}{1-t^2} \to -\ta. \label{1-formChangeVar3}
\end{alignat}

\begin{ex} \label{eg-single-ga-gb}
For a single $\ga$-block or $\gb^\star$-block, for any $s\in\N$ we have  by
 \eqref{endga-block} and \eqref{endgb-block})
\begin{align*}
\Xi(2n;s)=&\, \sum_{n\ge 0} \frac{a_{n}}{(2n)^s}
=\int_0^{\pi/2} F_s (\csc t-\cot t)\,dt
=(-1)^s \int_0^1 (2\tx_{-1}-\tx_{-i}-\tx_{i})\ty^{s-1},\\
\Xi(2n+1;s)=&\,\sum_{n\ge 0} \frac{a_{n}}{(2n+1)^s}
=\int_0^{\pi/2} F_s\,dt
=(-1)^s i\int_0^1  \td_{-i,i}\ty^{s-1}   
\end{align*}
by \eqref{1-formChangeVar1}--\eqref{1-formChangeVar3}.
Here we point out that after change of variable \eqref{equ-changVarSeq}
we need to multiply $(-1)^w$
when reversing the interval back to $[0,1]$, where $w$ is the weight of the value,
i.e., the number of 1-forms in the iterated integral. In particular, when $s\le 3$ we have
\begin{alignat}{5}\label{equ-Xi(2n;1)}
&\Xi(2n;1) =\log 2, \quad  &&\Xi(2n;2) =\frac{\pi^2-12\log^2 2}{24},  \quad && \Xi(2n;3) =\frac{4\log^3 2-\pi^2\log 2+6\ze(3)}{24},\\
&\Xi(2n+1;1) =\frac{\pi}{2}, \quad  &&\Xi(2n+1;2) =\frac{\pi\log 2}{2},   \quad  &&\Xi(2n+1;3) =\frac{\pi}{48}\Big(\pi^2 + 12\log^2 2\Big). \label{equ-Xi(2n+1;1)}
\end{alignat}
\end{ex}

\begin{ex}
For a double $\gb^\star$-block, for any $s\in\N$ we have by \eqref{midgbgb-block} and \eqref{endgb-block}
 \begin{align*}
\Xi(2n+1,2n+1;s,1)=&\,\sum_{n_1\ge n_2\ge 0} \frac{a_{n_1}}{(2n_1+1)^s(2n_2+1)}\\
=&\,\int_0^{\pi/2} F_s(\cot t\,dt \,dt+dt\,\tan t\,dt)
= (-1)^{s-1} i\int_0^1  (\td_{-i,i}\ty+\tz\td_{-i,i})\ty^{s-1} .  
\end{align*}
When $s\le 3$ we have
\begin{align*}
\Xi(2n+1,2n+1;1,1) =&\, \pi\log 2, \qquad \Xi(2n+1,2n+1;2,1)=\frac{3\pi\log^2 2}{4},   \\
\Xi(2n+1,2n+1;3,1)=&\, \frac{\pi}{48}\Big(\pi^2 \log 2 + 16\log^3 2+ 3\zeta(3)\Big).
\end{align*}
\end{ex}

\begin{ex} We now consider series given by a $\gb^\star$-$\ga$-block chain. For any $s\in\N$ we have
 \begin{align*}
\Xi(2n+1,2n;s,1)=&\,\sum_{n_1\ge n_2> 0} \frac{a_{n_1}}{(2n_1+1)^s (2n_2)}
=\int_0^{\pi/2} F_s\,dt\circ\sec t\sum_{n_2> 0}\frac{a_{n_2}(\sin t)}{2n_2}\quad(\text{by }\eqref{1stgbga-block})\\
=&\,\int_0^{\pi/2} F_s dt (\sec t\csc t \,dt-\csc t \,dt)
= (-1)^s i\int_0^1 (\ta+2\tx_{-1})\td_{-i,i}\ty^{s-1}.   
\end{align*}
by  \eqref{endga-block-Depth1}.
When $s=1$ we have
\begin{align*}
\Xi(2n+1,2n;1,1) =&2 \Im(2\Li_{1,1}(-1,-i)-\Li_2(i) )=2G - \frac12 \pi\log 2 \approx 0.7431381432,\\
\Xi(2n+1,2n;2,1) =&\frac{3\pi^3}{32}-4\Im\Li_3\Big(\frac{1+i}2\Big)-\frac{3\pi\log^22}8 \approx 0.0605084383,
\end{align*}
and by \eqref{equ-Xi(2n;1)} and   \eqref{equ-Xi(2n+1;1)}
\begin{align*}
 \sum_{n_1> n_2> 0} \frac{a_{n_1}}{(2n_1+1) (2n_2)}
=&\,\Xi(2n+1,2n;1,1)-\sum_{n>0} \frac{a_{n}}{(2n+1) (2n)}\\
=&\, \Xi(2n+1,2n;1,1)-\log 2+\frac{\pi}{2}-1\\
=&\,2G - \frac12\pi\log 2-\log 2 + \frac12\pi-1 \approx 0.6207872894.
\end{align*}
\end{ex}

Let $H_n$ be the $n$th harmonic number. By combining the examples from above we see that
\begin{align*}
\sum_{n\ge 0} \frac{a_{n}H_{2n}}{2n+1}
=\Xi(2n+1,2n;1,1)+\Xi(2n+1,2n+1;1,1)- \Xi(2n+1;2)=2G
\end{align*}
which is consistent with a formula on \cite[p.~10]{Chen2016}. Similarly,
all the four formulas on the top of \cite[p.~10]{Chen2016} can be verified.

\begin{ex} \label{eg-Xi(2n,2n+1)}
We now consider series given by a $\ga$-$\gb^\star$-block chain.
By \eqref{midgaga-block}
and  \eqref{1stgagb-block}
we have
\begin{align*}
\Xi(2n,2n+1;s,1)=&\,\sum_{n_1>n_2\ge 0} \frac{a_{n_1}}{(2n_1)^s(2n_2+1)} =\int_0^{\pi/2} F_s \csc t\,dt \tan t \,dt
=  (-1)^{s+1}\int_0^1 \tz\td_{-1,1} \ty^{s-1}
\end{align*}
by \eqref{1-formChangeVar1}--\eqref{1-formChangeVar3}.
So
\begin{align*}
\Xi(2n,2n+1;1,1)=&\, \int_0^1 \tz\td_{-1,1}=\frac{\pi^2}{8} \approx 1.2337005, \\
\Xi(2n,2n+1;2,1)=&\, -\int_0^1 \tz\td_{-1,1} \ty =\frac18 (7 \zeta(3)-\pi^2\log 2) \approx 0.196663732.
\end{align*}
\end{ex}

The next two examples show that, as predicted by Thm.~\ref{thm-Mixed-Parity-all},
if $\gam$-block does not appear then all the CMZVs should have the same weight as that of the Ap\'ery-type series.

\begin{ex}
For a $\gb^\star$-$\ga$-$\ga$-block chain,
by  \eqref{1stgbga-block}, \eqref{midgaga-block} and \eqref{endga-block-Depth1},
we have
\begin{align*}
&\, \Xi(2n+1,2n,2n;s,1,1) =\sum_{n_1\ge n_2> n_3>0} \frac{a_{n_1}}{(2n_1+1)^s (2n_2)(2n_3)}\\
=&\int_0^{\pi/2} F_s\,dt (\sec t\csc t \,dt )\big((\sec t\csc t-\csc t) \,dt \big)
-\int_0^{\pi/2}  F_s\, (\csc t \,dt) \big((\csc t-\cot t)\,dt\big)\\
=&\, i(-1)^s \int_0^1 (\ta+2\tx_{-1})  (\ta+\tx_{-1}+\tx_{1}) \td_{-i,i}\ty^{s-1}
-i(-1)^s \int_0^1(2\tx_{-1}-\tx_{-i}-\tx_{i}) \td_{-1,1} \td_{-i,i}\ty^{s-1}
\end{align*}
by \eqref{1-formChangeVar1}--\eqref{1-formChangeVar3}. Thus
 \begin{align*}
 \Xi(2n+1,2n,2n;1,1,1)  =&\, -i\int_0^1 \Big[(\ta+2\tx_{-1}) (\ta+\tx_{-1}+\tx_{1}) -(2\tx_{-1}-\tx_{-i}-\tx_{i}) \td_{-1,1} \Big]\td_{-i,i}\\
=&\, 2 G\log 2-\frac{\pi^3}6 + 8 \Im\Li_3\Big(\frac{1+i}2\Big)  \approx 0.6627044147, \\
\Xi(2n+1,2n,2n;2,1,1)  =&\, i\int_0^1 \Big[(\ta+2\tx_{-1}) (\ta+\tx_{-1}+\tx_{1}) -(2\tx_{-1}-\tx_{-i}-\tx_{i}) \td_{-1,1} \Big] \td_{-i,i}\ty\\
=&\, 26 \gb(4) - \frac{25}{96} \pi^3\log 2 -32\Im\Li_4\Big(\frac{1+i}2\Big) - 4\Im\Li_3\Big(\frac{1+i}2\Big)\log 2  \\
&\,  + \pi \Big(\frac5{24}\log^3 2 - \frac7{16} \zeta(3)\Big) \approx  0.012200331,
\end{align*}
where $\beta$ is the Dirichlet beta function
\begin{align}\label{DirichletBeta}
 \gb(s)=\sum_{k\ge 0}\frac{(-1)^k}{(2k+1)^s}.
\end{align}
\end{ex}

\begin{ex}
For a $\gb^\star$-$\ga$-$\gb^\star$-block chain, by  \eqref{1stgbga-block}, \eqref{midgagb-block} and \eqref{endgb-block-Depth1},
we have
 \begin{align*}
&\Xi(2n+1,2n,2n+1;s,1,1) =\,\sum_{n_1\ge n_2> n_3\ge 0} \frac{a_{n_1}}{(2n_1+1)^s (2n_2)(2n_3+1)}\\
=&\int_0^{\pi/2} F_s\,dt (\sec t\csc t\, dt) (\tan t \,dt)
= i(-1)^s \int_0^1 \tz (\ty+\tz) \td_{-i,i}\ty^{s-1}
\end{align*}
by \eqref{1-formChangeVar1}--\eqref{1-formChangeVar3}.
Thus we have
 \begin{align*}
&\,\Xi(2n+1,2n,2n+1;1,1,1)= -i \int_0^1 \tz (\ty+\tz) \td_{-i,i}  =\frac{\pi^3}{24}    \approx 1.29192819, \\
&\,\Xi(2n+1,2n,2n+1;2,1,1)= i \int_0^1 \tz (\ty+\tz) \td_{-i,i}\ty = \frac1{24}\pi^3\log 2-\frac7{32}\pi\zeta(3) \approx 0.0694147623, \\
&\,\Xi(2n+1,2n,2n+1;3,1,1)= -i \int_0^1 \tz (\ty+\tz) \td_{-i,i}\ty^2 \\
=&\, \frac{\pi}{11520}\bigg(73\pi^4 + 480\pi^2 \log^2 2 -
   120\Big (2\log^4 2 + 48 \Li_4\Big(\frac12\Big) + 63\log 2 \zeta(3) \Big)\bigg) \approx 0.0141472578.
\end{align*}
\end{ex}

\begin{ex}\label{eg-gam-ga-block-App}
We now consider series given by a $\gam$-$\ga$-block chain. By Thm.~\ref{thm-LeadGammaBlock}
we have
\begin{align*}
\Xi(2n-1,2n;1,1)=\Xi(2n;1)=\log 2 \quad\text{and} \quad \Xi(2n-1,2n;1,b)=\Xi(2n;b) \ \forall b\in\N.
\end{align*}
Next,  for any $s\ge 2$ by \eqref{equ-gam-depthd} and \eqref{endga-block-Depth1}
we get
\begin{align*}
\Xi(2n-1,2n;s,1)=&\,\sum_{n_1>n_2>0} \frac{a_{n_1}}{(2n_1-1)^s(2n_2)}
=\int_0^{\pi/2} F_{s-1} \frac{dt}{\tan^2 t} \tan t \sec t \, dt (\csc t-\cot t)\,dt\\
=&\,-\Xi(2n-1,2n;s-1,1)+\int_0^{\pi/2} F_{s-1}\,  ( \sec t \, dt -dt \,d(\sec t) \,(\csc t-\cot t)\,dt\\
=&\,-\Xi(2n-1,2n;s-1,1)+\int_0^{\pi/2} F_{s-1}\,  dt  \,(\sec t \csc t- \csc t)\,dt   \\
=&\,-\Xi(2n-1,2n;s-1,1)-(-1)^s i\int_0^1 (\ta+2\tx_{-1})\td_{-i,i}\ty^{s-2} \\
=&\,(-1)^s \Xi(2n-1,2n;2,1)-(-1)^s \sum_{j=1}^{s-2}  i\int_0^1 (\ta+2\tx_{-1})\td_{-i,i}\ty^j\\
=&\,(-1)^s \int_0^1 (2\tx_{-1}-\tx_i-\tx_{-i})- (-1)^s \sum_{j=0}^{s-2} i\int_0^1 (\ta+2\tx_{-1})\td_{-i,i}\ty^j.
\end{align*}
Note the highest CMZV weight drops by 1 as predicted by Thm.~\ref{thm-Mixed-Parity-all}\ref{enu:thm-Mixed-Parity-gen}.
For example,
\begin{align*}
&\,\Xi(2n-1,2n;2,1)=  2 G-\frac12\pi\log 2-\log 2 \approx 0.04999096264,\\
&\,\Xi(2n-1,2n;3,1)= \frac{3\pi^3}{32} -4\Im\Li_3\Big(\frac{1+i}2\Big) -\frac{\pi\log 2}{8} \Big(3\log2 -4\Big)+\log 2-2G\approx 0.010517475685.
\end{align*}
Next, for all $b\ge 2$ we have
\begin{align*}
\Xi(2n-1,2n;2,b)=&\,\sum_{n_1>n_2>0} \frac{a_{n_1}}{(2n_1-1)^2(2n_2)^b}
=\int_0^{\pi/2}  \frac{dt}{\tan^2 t} d(\sec t) \,\cot t\,dt \, F_{b-1}  (\csc t-\cot t)\,dt\\
=&\, \int_0^{\pi/2} \Big( \csc t\cot t \, dt\, \cot t\,dt - \cot^2 t \, dt\csc t\,dt\Big)  \, F_{b-1} (\csc t-\cot t)\,dt\\
=&\, \int_0^{\pi/2} \Big( (\csc t-1)\cot t\,dt - (\csc^2 t-1) \, dt\csc t\,dt\Big) \, F_{b-1}  (\csc t-\cot t)\,dt\\
=&\, \int_0^{\pi/2} \Big(   dt\csc t\,dt- \cot t\,dt\Big)  \, F_{b-1}  (\csc t-\cot t)\,dt\\
=&\, -i(-1)^b \int_0^1 (2\tx_{-1}-\tx_i-\tx_{-i}) \ty^{b-2} \td_{-1,1} \td_{-i,i} - (-1)^b \int_0^1 (2\tx_{-1}-\tx_i-\tx_{-i}) \ty^{b-1}.
&\,
\end{align*}
More generally, setting $\Xi'(s,b)=\Xi(2n-1,2n;s,b)$, we see that
for all $s\ge 3,b\ge 2$,
\begin{align*}
\Xi(2n-1,2n;s,b)=&\,\sum_{n_1>n_2>0} \frac{a_{n_1}}{(2n_1-1)^s(2n_2)^b}
=\int_0^{\pi/2} F_{s-1} \frac{dt}{\tan^2 t} \tan t \sec t \, dt \, F_b (\csc t-\cot t)\,dt\\
=&\,-\Xi'(s-1,1)+\int_0^{\pi/2} F_{s-1}\,  \Big( \sec t \, dt -dt \,d(\sec t)\Big) \, F_b (\csc t-\cot t)\,dt\\
=&\,-\Xi'(s-1,b)+\int_0^{\pi/2} F_{s-1}\,  dt  \, \sec t F_b (\csc t-\cot t)\,dt\\
=&\,-\Xi'(s-1,b)+\int_0^{\pi/2} F_{s-1}\,  dt  \, \csc t \,  dt  \, F_{b-1} (\csc t-\cot t)\,dt \\
=&\,-\Xi'(s-1,b)-(-1)^{s+b} i \int_0^1 (2\tx_{-1}-\tx_i-\tx_{-i})\ty^{b-2}\td_{-1,1} \td_{-i,i} \ty^{s-2}\\
=&\,(-1)^s \Xi'(2,b)-(-1)^{s+b}  i \sum_{j=1}^{s-2} \int_0^1 (2\tx_{-1}-\tx_i-\tx_{-i})\ty^{b-2}\td_{-1,1} \td_{-i,i} \ty^j\\
=&\,-(-1)^{s+b} \int_0^1 (2\tx_{-1}-\tx_i-\tx_{-i}) \bigg(\ty^{b-1}+ i \sum_{j=0}^{s-2}\ty^{b-2}\td_{-1,1} \td_{-i,i} \ty^j \bigg)\\
&\, \in\CMZV_{\le s+b-1}^4 \otimes \Q[i]
\end{align*}
by induction.
\end{ex}

\begin{ex}\label{eg-gam-gb-block-App}
We now consider series given by a $\gam$-$\gb^\star$-block chain. By  Thm.~\ref{thm-LeadGammaBlock} we have
\begin{align*}
\Xi(2n-1,2n+1;1,1)= \sum_{n_1>n_2\ge 0} \frac{a_{n_1}}{(2n_1-1)(2n_2+1)} = \Xi(2n+1;1)= \frac{\pi}{2}.
\end{align*}
For any $s\ge 2$ by \eqref{equ-gam-depthd} and \eqref{endgb-block-Depth1}
we get
\begin{align*}
\Xi(2n-1,2n+1;s,1)=&\,\sum_{n_1>n_2\ge 0} \frac{a_{n_1}}{(2n_1-1)^s(2n_2+1)}\\
=&\,\int_0^{\pi/2} F_{s-1} \frac{dt}{\tan^2 t} \sec^2 t \, dt \,dt
= \int_0^{\pi/2} F_s \,dt-I_{s-1},
\end{align*}
where for all $s\ge 2$
\begin{align*}
I_s=&\,\int_0^{\pi/2}  F_s \frac{dt}{\tan^2 t}  \tan t \,dt
=\int_0^{\pi/2}  F_s \csc^2 t dt \tan t \,dt - F_s \,dt \tan t \,dt
=\int_0^{\pi/2}  F_s \,dt - F_s \,dt \tan t \,dt -I_{s-1}.
\end{align*}
Note that
\begin{align*}
I_1=&\,\int_0^{\pi/2} \frac{dt}{\tan^2 t}  \tan t \,dt=\int_0^{\pi/2} \csc^2 t dt \tan t \,dt-\int_0^{\pi/2} dt \tan t \,dt \\
=&\, \int_0^{\pi/2} \Big[-\cot u\Big]_t^{\pi/2} \tan t \,dt-\int_0^{\pi/2} dt \tan t \,dt \\
=&\, \int_0^{\pi/2} dt-\int_0^{\pi/2} dt \tan t \,dt =\frac{\pi}{2}-i\int_0^1 \tz\td_{-i,i}.
\end{align*}
Thus
\begin{align*}
\Xi(2n-1,2n+1;s,1)=&\, i\int_0^1 \sum_{j=2}^s (-1)^{s-j}\Big(F_j \,dt + F_{j-1} \,dt \tan t \,dt\Big)-(-1)^{s}\frac{\pi}{2}\\
=&\,  (-1)^s \bigg(i\int_0^1\sum_{j=2}^s\Big(\td_{-i,i} \ty^{j-1}+\tz\td_{-i,i}\ty^{j-2}\Big)-\frac{\pi}{2}\bigg).
\end{align*}
Thus the highest CMZV weight again drops by 1 as predicted by  Thm.~\ref{thm-Mixed-Parity-all}\ref{enu:thm-Mixed-Parity-gen}.
In particular,
\begin{align*}
\Xi(2n-1,2n+1;2,1)=&\,  i\int_0^1 \td_{-i,i} \ty+\tz \td_{-i,i}-\frac{\pi}{2}=\pi\Big(\log2-\frac12\Big)\approx 0.606789764,\\
\Xi(2n-1,2n+1;3,1)=&\,  -i\int_0^1\Big(\td_{-i,i} (\ty^2+\ty+1)+ \tz\td_{-i,i}(\ty+1)\Big) ,\\
=&\,  \frac{\pi(2 - 4\log2 + 3\log^2 2)}{4}\approx 0.52525098,\\
\Xi(2n-1,2n+1;4,1)=&\,  i\int_0^1\Big(\td_{-i,i} (\ty^3+\ty^2+\ty+1)+ \tz\td_{-i,i}(\ty^2+\ty+1)\Big) ,\\
=&\,  \frac{\pi(48\log2+ \pi^2 \log2 - 36\log^2 2 + 16 \log^3 2  -24 + 3\zeta(3))}{48}  \approx  0.5072631333.
\end{align*}
Next, for all $b\ge 2$ we have
\begin{align*}
\Xi(2n-1,2n+1;2,b)=&\,\sum_{n_1>n_2>0} \frac{a_{n_1}}{(2n_1-1)^2(2n_2+1)^b}
=\int_0^{\pi/2}  \frac{dt}{\tan^2 t} d(\tan t) \,\cot t\,dt \, F_{b-1}  \,dt\\
=&\, \int_0^{\pi/2} \Big(\cot t \, dt\, \cot t\,dt - \cot^2 t \, dt \,dt\Big)  \, F_{b-1}  \,dt\\
=&\, \int_0^{\pi/2} \Big(\cot t \, dt\, \cot t\,dt - (\csc^2 t-1) \, dt \,dt\Big) \, F_{b-1}  \,dt\\
=&\, \int_0^{\pi/2} \Big(\cot t \, dt\, \cot t\,dt+dt \,dt-\cot t \,dt\Big) \, F_{b-1}   \,dt\\
=&\,i(-1)^b \int_0^1 \td_{-i,i}\ty^{b-2}\Big(\td_{-i,i}^2 - \ty^2- \ty) .
\end{align*}
In particular
\begin{align*}
\Xi(2n-1,2n+1;2,2)= &\,  \frac{\pi}{24} \Big(\pi^2 + 6\log2 (\log2-2) \Big) \approx 0.58048206459.
\end{align*}
More generally, setting $\Xi''(s,b)=\Xi(2n-1,2n+1;s,b)$ for all $s\ge2$, we see that
for all $s,b\ge 2$,
\begin{align*}
\Xi(2n-1,2n+1;s,b)=&\,\sum_{n_1>n_2\ge 0} \frac{a_{n_1}}{(2n_1-1)^s(2n_2+1)^b}
=\int_0^{\pi/2} F_{s-1} \frac{dt}{\tan^2 t} d(\tan t) \, F_b \,dt\\
=&\,-\Xi''(s-1,b)+\int_0^{\pi/2} F_{s-1}\,  \Big(\csc t \sec t \, dt -dt \,d(\tan t)\Big) \, F_b  \,dt\\
=&\,-\Xi''(s-1,b)+\int_0^{\pi/2} F_{s-1}\, \Big(\csc t \sec t \, dt -\tan t\,dt +dt\circ \tan t \Big) \, F_b  \,dt\\
=&\,-\Xi''(s-1,b)+\int_0^{\pi/2} F_{s+b-1}\,  dt -F_{s-1}\,dt\,dt\, F_{b-1}  \,dt \\
=&\,(-1)^{s}\Xi''(2,b)+i (-1)^{s+b}  \sum_{j=1}^{s-2} \int_0^1 \td_{-i,i}\Big( \ty^{b-2} \td_{-i,i}^2 \ty^j-\ty^{j+b} \Big) \\
=&\,-i(-1)^{s+b}\int_0^1 \td_{-i,i}\ty^{b-2}  \ty+i (-1)^{s+b}  \sum_{j=0}^{s-2} \int_0^1
    \td_{-i,i} \Big(\ty^{b-2} \td_{-i,i}^2 \ty^j- \ty^{j+b} \Big) \\
\in&\, i\CMZV_{\le s+b-1}^4.
\end{align*}
In particular,
\begin{align*}
\Xi(2n-1,2n+1;3,2)= \frac{\pi}{96}\Big(4 \pi^2 (\log2-1) + 48 \log2 - 24 \log^22 +    8 \log^32 + 3 \zeta(3)\Big) \approx 0.5202116317858.
\end{align*}
\end{ex}

\begin{ex}
We now consider series given by a $\gam$-$\gam$-$\gb^\star$-block chain. In general, we have
\begin{align*}
&\,\Xi(2n-1,2n-1,2n+1;a,b,c)\\
=&\,\sum_{n_1>n_2>n_3\ge 0} \frac{a_{n_1}}{(2n_1-1)^a(2n_2-1)^b(2n_3+1)^c}\\
=&\,\sum_{n_1>n_2+1>n_3\ge 0} \frac{a_{n_1}}{(2n_1-1)^a(2n_2+1)^b(2n_3+1)^c}\\
=&\,\sum_{n_1>n_2+1>n_3\ge 0} \frac{a_{n_1}}{(2n_1-1)^a(2n_2+1)^b(2n_3+1)^c}\\
=&\,\sum_{n_1>n_2+1,n_2\ge n_3\ge 0} \frac{a_{n_1}}{(2n_1-1)^a(2n_2+1)^b(2n_3+1)^c}\\
=&\,\bigg( \sum_{n_1>n_2,n_2\ge n_3\ge 0}  \
-\ \sum_{n_1=n_2+1,n_2\ge n_3\ge 0} \bigg) \frac{a_{n_1}}{(2n_1-1)^a(2n_2+1)^b(2n_3+1)^c}\\
=&\,\sum_{n_1>n_2\ge n_3\ge 0} \frac{a_{n_1}}{(2n_1-1)^a(2n_2+1)^b(2n_3+1)^c}
-\sum_{n_2\ge n_3\ge 0} \frac{a_{n_1}}{(2n_2+1)^{a+b}(2n_3+1)^c}.
\end{align*}
We are thus reduced to the case with the only $\gam$-block appearing at the beginning, as demonstrated in the proof
of  Thm.~\ref{thm-Mixed-Parity-all}.

In fact, we may also compute these sums directly.
By  Thm.~\ref{thm-LeadGammaBlock}   we have
\begin{align*}
\Xi(2n-1,2n-1,2n+1;1,1,1)= \Xi(2n+1;1)= \frac{\pi}{2}.
\end{align*}
For any $s\ge 2$ by \eqref{equ-gam-depthd}  and \eqref{endgb-block-Depth1}
we get
\begin{align*}
\Xi(2n-1,2n-1,2n+1;s,1,1)=&\,\sum_{n_1>n_2>n_3\ge 0} \frac{a_{n_1}}{(2n_1-1)^s(2n_2-1)(2n_3+1)}\\
=&\,\int_0^{\pi/2} F_{s-1} \frac{dt}{\tan^2 t} \tan t dt \sec^2 t \, dt \,dt.
\end{align*}
If $s=2$, we get
\begin{align*}
&\, \Xi(2n-1,2n-1,2n+1;2,1,1)=\sum_{n_1>n_2>n_3\ge 0} \frac{a_{n_1}}{(2n_1-1)^2(2n_2-1)(2n_3+1)}\\
=&\  \int_0^{\pi/2} \frac{dt}{\tan^2 t} \tan t \,dt \sec^2 t \, dt \,dt
=\int_0^{\pi/2} \frac{dt}{\tan^2 t} \Big(\tan^2 t \, dt \,dt -\tan t \,dt \tan t \,dt\Big)\\
=&\ \int_0^{\pi/2} \frac{dt}{\tan^2 t} \Big(\sec^2 t \, dt \,dt -dt \,dt -\tan t \,dt \tan t \,dt\Big)\\
=&\, \int_0^{\pi/2} \cot t\,dt  \,dt - (\csc^2 t-1)\,dt \Big(\tan t \, dt+dt \,dt  + \tan t \,dt \tan t \,dt\Big) \\
=&\,  \int_0^{\pi/2} \cot t\,dt  \,dt +\int_0^{\pi/2} \Big[\cot u\Big]_t^{\pi/2} \Big(\tan t \, dt+dt \,dt  + \tan t \,dt \tan t \,dt\Big) \\
&\, \hskip3cm +\int_0^{\pi/2}dt\Big(\tan t \, dt+dt \,dt  + \tan t \,dt \tan t \,dt\Big)  \\
=&\,  \int_0^{\pi/2} \cot t\,dt  \,dt-\int_0^{\pi/2} \Big(dt+\cot t\,dt \,dt + dt \tan t \,dt\Big) +\int_0^{\pi/2}dt\Big(\tan t \, dt+dt \,dt  + \tan t \,dt \tan t \,dt\Big)  \\
=&\,   \int_0^{\pi/2}dt\Big(dt \,dt  + \tan t \,dt \tan t \,dt-1\Big)  \\
=&\,  i\int_0^1   \td_{-i,i}^3 + \td_{-i,i} - \tz \tz \td_{-i,i}
=\frac{\pi(\pi^2 +6\log^2 2-12)}{24}  \approx 0.0984787829.
\end{align*}
If $s\ge 3$ then
\begin{align*}
&\,\Xi(2n-1,2n-1,2n+1;s,1,1)
=\int_0^{\pi/2} F_{s-1} \, d(-\cot t)\, \tan t dt \,d(\tan t) \,dt-F_{s-1} \, dt\, \tan t dt \, d(\tan t) \, dt \\
=&\,-\Xi(2n-1,2n-1,2n+1;s-1,1,1)+\int_0^{\pi/2} F_{s-1} (dt- dt\, \tan t dt)\, d(\tan t) \,dt \\
=&\,-\Xi(2n-1,2n-1,2n+1;s-1,1,1)+\int_0^{\pi/2} F_{s-1} \Big[( \tan t\, dt- dt\, \tan^2 t dt)\,dt- (dt- dt\, \tan t dt)\,\tan t\,dt\Big] \\
=&\,-\Xi(2n-1,2n-1,2n+1;s-1,1,1)+\int_0^{\pi/2} F_{s-1} \Big[ dt\,dt\,dt+ dt\, \tan t \,dt\,\tan t\,dt\Big] \\
=&\,-\Xi(2n-1,2n-1,2n+1;s-1,1,1)- (-1)^s i\int_0^1 \Big[\tz\tz \td_{-i,i}-\td_{-i,i}^3 \Big] \ty^{s-2} \\
=&\,(-1)^s \Xi(2n-1,2n-1,2n+1;2,1,1)- (-1)^s \sum_{j=1}^{s-2} i\int_0^1 \Big[\tz\tz\td_{-i,i}-\td_{-i,i}^3 \Big] \ty^j\\
=&\,(-1)^s i \td_{-i,i}-(-1)^s \sum_{j=0}^{s-2} i\int_0^1 \Big[\tz\tz\td_{-i,i}-\td_{-i,i}^3 \Big] \ty^j.
\end{align*}
Thus the highest CMZV weight again drops by 1 as predicted by  Thm.~\ref{thm-Mixed-Parity-all}\ref{enu:thm-Mixed-Parity-gen}.
In particular
\begin{align*}
&\, \Xi(2n-1,2n-1,2n+1;3,1,1) =\frac{\pi(2\pi^2\log2-4\pi^2+24\log^3 2 -24\log^2 2-15\zeta(3)+48)}{96}  \approx 0.02076805749.
\end{align*}
\end{ex}

\begin{ex}
We now consider series given by a $\ga$-$\gam$-block chain. In general, for all $a,b\in\N$
\begin{align*}
\Xi(2n,2n-1;s,b)=&\,\sum_{n_1> n_2> 0} \frac{a_{n_1}}{(2n_1)^s (2n_2-1)^b}\\
=&\,\sum_{n_1> n_2+1> 0} \frac{a_{n_1}}{(2n_1)^s (2n_2+1)^b}\\
=&\,\sum_{n_1> n_2\ge 0} \frac{a_{n_1}}{(2n_1)^s (2n_2+1)^b}-\sum_{n>0} \frac{a_{n}}{(2n)^s (2n-1)^b} .
\end{align*}
By partial fractions we get
\begin{align*}
\frac{1}{x^{a}(x-1)^b}=\sum_{j=1}^b \binom{-a}{b-j} \frac{1}{(x-1)^j} +\sum_{j=1}^a (-1)^{a+b-j} \binom{-b}{a-j} \frac{1}{x^j}.
\end{align*}
For example, when $b=1$ we have
\begin{align*}
\Xi(2n,2n-1;s,1)=&\,\sum_{n_1> n_2\ge 0} \frac{a_{n_1}}{(2n_1)^s (2n_2+1)}
-\sum_{n>0} \frac{a_{n}}{2n-1}+\sum_{j=1}^s \sum_{n>0} \frac{a_{n}}{(2n)^j} .
\end{align*}
From Example~\ref{exa-gam1}, Example~\ref{eg-Xi(2n,2n+1)} and Example~\ref{eg-single-ga-gb}
\begin{align*}
\Xi(2n,2n-1;s,1)=&\,  (-1)^{s-1}  \int_0^1 \tz\td_{-1,1}\ty^{s-1}
-1-\sum_{j=0}^{s-1} (-1)^j \int_0^1 (2\tx_{-1}- \tx_i-\tx_{-i})\ty^j.
\end{align*}

We can also compute $\Xi(2n,2n-1;s,1)$ directly as follows. By \eqref{equ-gam-d=1} and   \eqref{equ-ga-depthd}
we have
\begin{align*}
\Xi(2n,2n-1;s,1)=&\,\sum_{n_1> n_2> 0} \frac{a_{n_1}}{(2n_1)^s (2n_2-1)}\\
=&\,\int_0^{\pi/2} F_s ( \sin t\, dt-\csc t\,dt \tan t \,dt)  \tan t \sec t \,dt  \\
=&\,\int_0^{\pi/2} F_s ( \tan t\, dt-\csc t\,dt \tan t  \sec t\,dt)   -\int_0^{\pi/2} F_s ( \sin t\, dt-\csc t\,dt \tan t \,dt)    \\
=&\,\int_0^{\pi/2} F_s ( \tan t\, dt-\csc t\sec t\,dt + \csc t\,dt)   -\int_0^{\pi/2} F_s ( \sin t\, dt-\csc t\,dt \tan t \,dt)    \\
=&\,  (-1)^{s-1}  \int_0^1 ( \tx_i+\tx_{-i}-2\tx_{-1}+ \tz\td_{-1,1} ) \ty^{s-1} - I_s,  
\end{align*}
where $I_1=\int_0^{\pi/2}  \sin t\, dt =1$ and for all $s\ge 2$
\begin{align*}
I_s=&\, \int_0^{\pi/2} F_s\sin t\, dt = \int_0^{\pi/2} F_s(1-\cos t)\, dt
=I_{s-1}+\int_0^{\pi/2} F_{s-1} (\cot t-\csc t\, dt \\
=&\, I_{s-1}+(-1)^s \int_0^1 (2\tx_{-1}- \tx_i-\tx_{-i})\ty^{s-2}
=1+\sum_{j=0}^{s-2} (-1)^j \int_0^1 (2\tx_{-1}- \tx_i-\tx_{-i})\ty^j.
\end{align*}
Thus
\begin{align*}
\Xi(2n,2n-1;s,1)=&\,  (-1)^{s-1}  \int_0^1 \tz\td_{-1,1}\ty^{s-1}
-1-\sum_{j=0}^{s-1} (-1)^j \int_0^1 (2\tx_{-1}- \tx_i-\tx_{-i})\ty^j  .
\end{align*}
In particular,
\begin{align*}
\Xi(2n,2n-1;1,1)=&\, \int_0^1 \Big( \tz\td_{-1,1}+\tx_i+\tx_{-i}-2\tx_{-1}\Big) -1=\frac{\pi^2}8 + \log 2-1 \approx 0.926847730696,\\
\Xi(2n,2n-1;2,1)=&\,\frac{1}{24}\Big(\pi^2 (1-3\log2) + 21\ze(3)\Big)-\frac12\log^2 2+\log2-1 \approx 0.0608179225954,\\
\Xi(2n,2n-1;3,1)=&\,\frac{\pi^2}{192}\Big(\pi^2+12\log^22-8\log2+8\Big)+\frac{\log^22}{6}\Big(\log2-3\Big)\\
    &\, \ \hskip2cm -\frac{\log2}{8}\Big(7\ze(3)-8\Big)+\frac{\ze(3)-4}{4}\approx 0.0097817.
\end{align*}
\end{ex}

\begin{ex} In this example, we explicit reduce a $\gam$-$\ga$-$\gam$-block chain to the case
with the only $\gam$-block appearing at the beginning. For all $a,b,c\in \N$,
\begin{align*}
&\,\Xi(2n-1,2n,2n-1;a,b,c)\\
=&\,\sum_{n_1>n_2>n_3> 0} \frac{a_{n_1}}{(2n_1-1)^a(2n_2)^b(2n_3-1)^c}\\
=&\,\sum_{n_1>n_2>n_3+1> 0} \frac{a_{n_1}}{(2n_1-1)^a(2n_2)^b(2n_3+1)^c}\\
=&\,\sum_{n_1>n_2>n_3+1,n_3\ge 0}\frac{a_{n_1}}{(2n_1-1)^a(2n_2)^b(2n_3+1)^c}\\
=&\,\bigg( \sum_{n_1>n_2>n_3\ge 0}  \
-\ \sum_{n_1>n_2=n_3+1, n_3\ge 0} \bigg) \frac{a_{n_1}}{(2n_1-1)^a(2n_2)^b(2n_3+1)^c}\\
=&\,\sum_{n_1>n_2> n_3\ge 0} \frac{a_{n_1}}{(2n_1-1)^a(2n_2)^b(2n_3+1)^c}
-\sum_{n_1>n_2>0} \frac{a_{n_1}}{(2n_1-1)^a(2n_2)^{b}(2n_2-1)^c}.
\end{align*}
By partial fractions we get
\begin{align*}
\frac{1}{x^{b}(x-1)^c}=\sum_{j=1}^c \binom{-b}{c-j} \frac{1}{(x-1)^j} +\sum_{j=1}^b (-1)^{b+c-j} \binom{-c}{b-j} \frac{1}{x^j}.
\end{align*}
For example, we have
\begin{align*}
\Xi(2n-1,2n,2n-1;a,2,2)
=&\,\sum_{n_1>n_2> n_3\ge 0} \frac{a_{n_1}}{(2n_1-1)^a(2n_2)^2(2n_3+1)^2}\\
&\,-\sum_{n_1>n_2>0} \frac{a_{n_1}}{(2n_1-1)^a}\bigg( \frac{1}{(2n_2)^2}-\frac{2}{2n_2}+\frac{1}{(2n_2-1)^2}-\frac{2}{2n_2-1}\bigg).
\end{align*}
But for any $a,b\in\N$
\begin{align*}
\sum_{n_1>n_2>0} \frac{a_{n_1}}{(2n_1-1)^a(2n_2-1)^b}=\sum_{n_1>n_2\ge 0} \frac{a_{n_1}}{(2n_1-1)^a(2n_2+1)^b}
-\sum_{n> 0} \frac{a_{n}}{(2n_1-1)^{a+b}}.
\end{align*}
\end{ex}

\section{Examples of Ap\'ery-type series involving squared central binomial coefficients}
In this section, we present a few examples illustrating the ideas in  Thm.~\ref{thm-binnSquare}.

\begin{ex}  \label{eg-binn-sqr-single-ga-block}
We consider series given by a $\ga$-block.
From  \eqref{equ-binnSquare-ga}, \eqref{endga-block} and \eqref{endga-block-Depth1}
we have
\begin{align*}
 \sum_{n> 0} \frac{a_{n}^2}{(2n)^s}
=\frac{2}{\pi} \int_0^{\pi/2} dt\, F_s(\csc t-\cot t)dt
=&\,  (-1)^{s} \frac{2i}{\pi} \int_0^1 (2\tx_{-1}-\tx_{-i}-\tx_{i}) \ty^{s-1} \td_{-i,i}.  
\end{align*}
Thus
\begin{align}\label{equ-anSqOver2n}
\sum_{n> 0} \frac{a_{n}^2}{2n}=&\, -\frac{2i}{\pi} \int_0^1 (2\tx_{-1}-\tx_{-i}-\tx_{i}) \td_{-i,i}
    =\frac{2}{\pi} (\pi\log2-2G)\approx 0.2200507,\\
\sum_{n> 0} \frac{a_{n}^2}{(2n)^2}
=&\, \frac{2i}{\pi} \int_0^1 (2\tx_{-1}-\tx_{-i}-\tx_{i})\ty\td_{-i,i}
    = \frac{2}{\pi}\left(\frac3{16}\pi^3 - 8\Im\Li_3\Big(\frac{1+i}{2}\Big) - \frac34 \pi \log^2 2\right)\approx 0.07704,\nonumber\\
\sum_{n> 0} \frac{a_{n}^2}{(2n)^3}
=&\, - \frac{2i}{\pi} \int_0^1 (2\tx_{-1}-\tx_{-i}-\tx_{i}) \ty^2 \td_{-i,i}\nonumber\\
=&\, \frac{2}{\pi}\left(24 \gb(4) - 32 \Im\Li_4\Big(\frac{1+i}{2}\Big) -
 \frac38 \pi^3 \log 2 + \frac{\pi}4 (2\log^3 2 + \zeta(3))\right)\approx 0.03418181.\nonumber
\end{align}
\end{ex}

\begin{ex}  We consider series given by a $\gb^\star$-block.
From \eqref{equ-binnSquare-gb}, \eqref{endgb-block} and \eqref{endgb-block-Depth1}
we have
\begin{align*}
 \sum_{n\ge 0} \frac{a_{n}^2}{(2n+1)^s}
=\frac{2}{\pi} \int_0^{\pi/2} \csc t\,dt\, F_s \, dt
=&\,  (-1)^{s+1} \frac{2i}{\pi} \int_0^1 \td_{-i,i} \ty^{s-1} \td_{-1,1}.  
\end{align*}
Thus
\begin{align*}
\sum_{n\ge 0} \frac{a_{n}^2}{2n+1}=&\, \frac{2i}{\pi} \int_0^1\td_{-i,i} \td_{-1,1} =\frac{4G}{\pi}\approx 1.1662436,\\
\sum_{n\ge 0} \frac{a_{n}^2}{(2n+1)^2}
=&\, -\frac{2i}{\pi} \int_0^1 \td_{-i,i} \ty \td_{-1,1}
    =  \frac{2}{\pi}\left(\frac3{16}\pi^3 - 8\Im\Li_3\Big(\frac{1+i}{2}\Big)+\frac{\pi}4 \log^2 2\right)\approx 1.0379477643, \nonumber\\
\sum_{n\ge 0} \frac{a_{n}^2}{(2n+1)^3}
=&\, \frac{2i}{\pi} \int_0^1 \td_{-i,i} \ty^2  \td_{-1,1}
= \frac{2}{\pi}\left(\Im\Li_4\Big(\frac{1+i}{2}\Big)-24\gb(4) +
 \frac38 \pi^3 \log 2 + \frac{\pi}6 \log^3 2\right) \approx 1.01087951. \nonumber
\end{align*}
\end{ex}

\begin{ex}  \label{binn-sqr-1-gam-block}
We consider series given by a single $\gam$-block. From \eqref{equ-gam-d=1}
we have
\begin{align}\label{equ-Gamma1}
 \sum_{n> 0} \frac{a_{n}^2}{2n-1} =\frac{2}{\pi}\int_0^{\pi/2} \cos t\,dt\, d(\sec t)
=\frac{2}{\pi}  \int_0^{\pi/2} (1-\cos t) dt
=\frac{2}{\pi} \Big(\frac{\pi}{2}-1\Big) \approx 0.36338.
\end{align}
Note that there is no weight drop here so it is one of the two reasons we need to
introduce $\iota(\emptyset)=1$ in Thm.~\ref{thm-binnSquare}.
The other reason is given in Example~\ref{eg-binn-sqr-gam-ga-block}.
For all $s\ge2$
\begin{align*}
\Gamma_s:=  \sum_{n> 0} \frac{a_{n}^2}{(2n-1)^s}
=&\,\frac{2}{\pi} \int_0^{\pi/2} \sin t\,dt\, F_{s-1} \frac{dt}{\tan^2 t} d(\sec t)\\
=&\,\frac{2}{\pi} \int_0^{\pi/2} \sin t\,dt\, F_{s-1} (\csc^2 t-1)dt\, d(\sec t)\\
=&\,\frac{2}{\pi} \int_0^{\pi/2} \sin t\,dt\, F_{s-1} \Big( d(-\cot t)\, d(\sec t)- dt\,d(\sec t)\Big)\\
=&\,-\Gamma_{s-1}+\frac{2}{\pi} \int_0^{\pi/2} \sin t\,dt\, F_{s-1} \Big( \sec t\,dt-dt\, d(\sec t)\Big)\\
=&\,-\Gamma_{s-1}+\frac{2}{\pi} \int_0^{\pi/2} \sin t\,dt\, F_{s-1}  dt
=-\Gamma_{s-1}+I_{s-1} \\
=&\, \frac{2}{\pi}\left( (-1)^{s}\int_0^{\pi/2} \sin t\,dt \frac{dt}{\tan^2 t} d(\sec t)+\sum_{j=2}^{s-1} (-1)^{s-1-j} I_j\right),
\end{align*}
where
\begin{align*}
I_j=&\, \int_0^{\pi/2} \sin t\,dt\, F_j\,dt= \int_0^{\pi/2} \cos t\cot t\,dt\, F_{j-1}\,dt=\int_0^{\pi/2} \csc t dt\, F_{j-1}\,dt-I_{j-1}\\
=&\,(-1)^{j-1}\int_0^{\pi/2} \sin t\,dt\,dt+\sum_{k=1}^{j-1} (-1)^{j-1-k} \int_0^{\pi/2} \csc t\,  dt\, F_k\,dt,\\
=&\,(-1)^{j-1}  +\sum_{k=1}^{j-1} (-1)^j i\int_0^1 \td_{-i,i}\ty^{k-1} \td_{-1,1}.
\end{align*}
Hence for all $s\ge2$ we get
\begin{align*}
\Gamma_s= \sum_{n> 0} \frac{a_{n}^2}{(2n-1)^s}
=&\,  (-1)^{s} \frac{2}{\pi}\left( s-\frac{\pi}{2}-\sum_{k=2}^{s-1} i(s-k)\int_0^1 \td_{-i,i}\ty^{k-2}  \td_{-1,1}\right).
\end{align*}
Thus
\begin{align*}
\sum_{n> 0} \frac{a_{n}^2}{(2n-1)^2}=&\,\frac2{\pi}\Big(2-\frac{\pi}2\Big)  \approx 0.27323954,\\
\sum_{n> 0} \frac{a_{n}^2}{(2n-1)^3}=&\,\frac{2}{\pi}\left(\frac{\pi}{2}+2G-3\right)\approx 0.256384299.
\end{align*}
\end{ex}

\begin{ex}
We consider series given by a $\gam^{\circ d}$-block chain, with weight equal to 1 for every block.
From \eqref{equ-binnSquare-gam1}
and \eqref{equ-gam-d=1}
we have
\begin{align*}
\Gamma^{\circ d}:=&\, \sum_{n_1>\cdots>n_d> 0} \frac{a_{n_1}^2}{(2n_1-1)\cdots (2n_d-1)}
=\frac{2}{\pi}\int_0^{\pi/2} \cos t\,dt\, (\tan t\,dt)^{d-1}\, d(\sec t)\\
=&\,\frac{2}{\pi}  \int_0^{\pi/2} \cos t\,dt\,  (\tan t\,dt)^{d-2}\, (\tan t\sec t-\tan t) dt\\
=&\,\Gamma^{\circ (d-1)}- \frac{2}{\pi} A_{d-1} =\Gamma^{\circ 1}-\frac{2}{\pi} \sum_{k=1}^{d-1}A_k,
\end{align*}
where $A_0=1$ and for all $d\ge1$
\begin{align}\label{defn-Ad}
A_d= \int_0^{\pi/2} \cos t\,dt\, (\tan t\, dt)^{d}
=&\,  A_{d-1}-(-1)^d  \int_0^1  \tz^{d-1} (\tx_{-i}+\tx_{i}) \nonumber\\
=&\, A_0+\sum_{j=0}^{d-1} (-1)^{j} \int_0^1  \tz^j (\tx_{-i}+\tx_{i}).
\end{align}
Observe that by \eqref{equ-Gamma1}
in Example~\ref{binn-sqr-1-gam-block} we get $\Gamma^{\circ 1}=\Gamma_1=1-\frac{2}{\pi}$.
Therefore, for all $d\ge 1$ we have
\begin{align*}
\Gamma^{\circ d}:=&\,\frac{2}{\pi} \bigg(\frac{\pi}{2}-\sum_{k=0}^{d-1}A_k\bigg)
=\frac{2}{\pi} \bigg(\frac{\pi}{2}-d-\sum_{j=0}^{d-2} (-1)^{j}(d-1-j) \int_0^1  \tz^j (\tx_{-i}+\tx_{i})\bigg).
\end{align*}
We see that the weight drops by 1 as predicted by  Thm.~\ref{thm-binnSquare}\ref{enu:thm-binnSquare-gen}.
In particular, when $d=2$
\begin{align*}
&\, \sum_{n_1>n_2> 0} \frac{a_{n_1}^2}{(2n_1-1)(2n_2-1)}=\frac{2}{\pi}\Big(\frac{\pi}{2} + \log 2-2\Big) \approx 0.16803165557.
\end{align*}
\end{ex}

\begin{ex} \label{gam-depthd-App}
We consider series given by a $\gam^{\circ d}$-block chain with leading weight two and trailing weight one blocks.
By \eqref{equ-gam-d=1}
and \eqref{equ-binnSquare-gam-depthd}
\begin{align*}
S_d:=&\, \sum_{n_1>n_2>\cdots>n_d> 0} \frac{a_{n_1}^2}{(2n_1-1)^2(2n_2-1)\cdots(2n_d-1)}
=\frac{2}{\pi}  \int_0^{\pi/2} \sin t\, dt  \cot^2 t\,dt\, (\tan t\,dt)^{d-1}\,d(\sec t) \\
=&\,\frac{2}{\pi}  \int_0^{\pi/2} \bigg(- \sin t\, dt \, dt\, (\tan t\,dt)^{d-1}  d(\sec t)
    -\cos  t\, dt(\tan t\,dt)^{d-1} \,d(\sec t) \\
 &\,\hskip1cm +\sin t\, dt  \, dt \, (\tan t\,dt)^{d-2} d(\sec t) \bigg)
    = -2X_d + X_{d-1}  \\
=&\, \frac{2}{\pi} \bigg(d+1-\frac{\pi}{2}
+2\sum_{j=0}^{d-2} (-1)^{j}(d-1-j) \int_0^1  \tz^j (\tx_{-i}+\tx_{i})
-\sum_{j=0}^{d-3} (-1)^{j}(d-2-j) \int_0^1  \tz^j (\tx_{-i}+\tx_{i})\bigg)\\
=&\, \frac{2}{\pi} \bigg(d+1-\frac{\pi}{2}
+\sum_{j=0}^{d-2} (-1)^{j}(d-j) \int_0^1  \tz^j (\tx_{-i}+\tx_{i}) \bigg).
\end{align*}
This shows the weight may drop by two in this special case. In particular,
\begin{align}\label{equ-S2}
S_2=&\,\frac{2}{\pi}\Big(3 -\frac{\pi}2-2\log2 \Big) \approx 0.0273169,\\
S_3=&\,\frac{2}{\pi}\Big(\log^2 2-\frac{\pi^2}{12}-2\log2-\frac{\pi}{2}+4\Big) \approx 0.00493260131.\notag
\end{align}
\end{ex}

\begin{ex} \label{gam-gam}
For a $\gam$-$\gam$-block chain with second $\gam$-block having weight 1, we have
\begin{align*}
T_2:=\sum_{n_1>n_2> 0} \frac{a_{n_1}^2}{(2n_1-1)^2(2n_2-1)}=S_2=\frac{2}{\pi}\Big(3 -\frac{\pi}2-2\log2 \Big)
\end{align*}
by \eqref{equ-S2}. For all $s\ge 3$, by  \eqref{equ-gam-d=1}
and  \eqref{equ-binnSquare-gam-depthd}
\begin{align*}
T_s:=&\, \sum_{n_1>n_2> 0} \frac{a_{n_1}^2}{(2n_1-1)^s(2n_2-1)}
=\frac{2}{\pi}  \int_0^{\pi/2} \sin t\, dt (\cot t\,dt)^{s-2}\,\cot^2 t\,dt\, \tan t\,dt\,d(\sec t) \\
=&\, -T_{s-1}+\frac{2}{\pi} (B_{s-2} - C_{s-2})
=(-1)^s T_2+\frac{2}{\pi} \sum_{k=1}^{s-2} (-1)^{s-k} (B_k - C_k),
\end{align*}
where
\begin{align*}
B_0=&\, \int_0^{\pi/2} \sin t\,dt \, dt \,  d(\sec t)  =\int_0^{\pi/2}\cos t(\sec t-1)\,dt=\frac{\pi}{2}-1   \\
B_s=&\, \int_0^{\pi/2} \bigg(\sin t\,dt (\cot t\,dt)^s \, dt \,  d(\sec t) \bigg)
=\int_0^{\pi/2} \bigg(\cos t\cot t\,dt (\cot t\,dt)^{s-1}  \, dt \,  d(\sec t) \bigg)  \\
=&\, -B_{s-1}+ \int_0^{\pi/2} \bigg(\csc t\,dt (\cot t\,dt)^{s-1} \, dt \,d(\sec t) \bigg)
\end{align*}
and similarly
\begin{align*}
C_0=&\, \int_0^{\pi/2} \bigg(\sin t\,dt  \, dt\, \tan t\,dt\,  d(\sec t) \bigg)
= \int_0^{\pi/2} \cos t\,dt (d(\sec t)-\tan t\,dt) \\
= &\, \int_0^{\pi/2} (1-\cos t-\tan t+ \sin t\tan t)\,dt  \\
= &\, \int_0^{\pi/2} (1-2\cos t+\sec t-\tan t)\,dt = \frac{\pi}{2}-2-\int_0^1 (\tx_{-i}+\tx_{i})
= \frac{\pi}{2}-2+\log 2, \\
C_s=&\, \int_0^{\pi/2} \bigg(\sin t\,dt (\cot t\,dt)^s\, dt\, \tan t\,dt\,  d(\sec t) \bigg)
= \int_0^{\pi/2} \bigg( (\cos t\cot t\,dt)\, (\cot t\,dt)^{s-1}\, dt\,  \tan t\,dt\,  d(\sec t) \bigg)  \\
=&\, \int_0^{\pi/2} \bigg( (\csc t-\sin t) dt  \, (\cot t\,dt)^{s-1} \, dt\, (\tan t\sec t-\tan t)\,dt  \bigg)  \\
=&\,-C_{s-1}+ \int_0^{\pi/2} \bigg(  \csc t\,dt\, (\cot t\,dt)^{s-1} \, dt\, \Big(d(\sec t)-\tan t\,dt \Big) \bigg)  \\
=&\,-C_{s-1}+ \int_0^{\pi/2} \bigg(  \csc t\,dt\, (\cot t\,dt)^{s-1} \, dt \,d(\sec t) \bigg)
-  \int_0^{\pi/2} \bigg(\csc t\,dt\, (\cot t\,dt)^{s-1} \, dt\, \tan t\,dt \bigg).
\end{align*}
Thus
\begin{align*}
B_s-C_s=&\, -(B_{s-1}-C_{s-1})+\int_0^{\pi/2} \bigg(  \csc t\,dt\, (\cot t\,dt)^{s-1} \, dt\, \tan t\,dt \bigg)  \\
=&\, -(B_{s-1}-C_{s-1})+(-1)^s i\int_0^1\tz\td_{-i,i}\ty^{s-1}\td_{-1,1}  \\
=&\, (-1)^s (B_0-C_0)+(-1)^s \sum_{j=0}^{s-1} i\int_0^1\tz\td_{-i,i}\ty^j \td_{-1,1}  .
\end{align*}
Since $B_0-C_0=1-\log 2$ we get
\begin{align*}
T_s=&\, (-1)^s T_2+\frac{2}{\pi} \sum_{k=1}^{s-2} (-1)^{s-k} \left(
 (-1)^k(1-\log 2) +(-1)^k \sum_{j=0}^{k-1} i\int_0^1\tz\td_{-i,i}\ty^j \td_{-1,1}  \right) \\
=&\, (-1)^s T_2+\frac{2}{\pi} \sum_{k=1}^{s-2} (-1)^{s-k} \left(
 (-1)^k(1-\log 2) +(-1)^k \sum_{j=0}^{k-1} i\int_0^1\tz\td_{-i,i}\ty^j \td_{-1,1}  \right) \\
=&\, (-1)^{s} \frac{2}{\pi} \left(3 -\frac{\pi}2-2\log2+
(s-2)(1-\log 2) + \sum_{j=0}^{s-3} i(s-2-j)\int_0^1\tz\td_{-i,i}\ty^j \td_{-1,1}  \right)\\
=&\, (-1)^{s} \frac{2}{\pi} \left(1 -\frac{\pi}2+
s(1-\log 2) + \sum_{j=0}^{s-3} i(s-2-j)\int_0^1\tz\td_{-i,i}\ty^j \td_{-1,1}  \right).
\end{align*}
In particular
\begin{align*}
T_3=&\, - \frac{2}{\pi} \left(4 -\frac{\pi}2-3\log 2 + i \int_0^1\tz\td_{-i,i} \td_{-1,1}  \right) \\
=&\,\frac{2}{\pi} \left( \frac{3}{16}\pi^3 - 8 \Im\Li_3\Big(\frac{1+i}2\Big) +
 \frac14\pi (2 + \log^2 2) -G \log^2 2 + \log^3 2 -4  \right)
\approx 0.064673585123, \\
T_4=&\,  \frac{2}{\pi} \left(5 -\frac{\pi}2-4\log 2 +2i\int_0^1\tz\td_{-i,i} \td_{-1,1}
+ i\int_0^1\tz\td_{-i,i}\ty \td_{-1,1}   \right) \\
=&\,\frac{2}{\pi} \left(64 \Im\Li_4\Big(\frac{1+i}2\Big) + 8 \Im\Li_3\Big(\frac{1+i}2\Big)\log 2 +
 \frac{1}{12} \pi ( \log^3 2-6 - 6 \log^2 2) - 50 \gb(4) \right.  \\
&\, \left. +16 \Im\Li_3\Big(\frac{1+i}2\Big)+
 \frac{3}{16} \pi^3 (\log^3 2-2) - \log^4 2 +G \log^4 2 + 5\right)
\approx 0.002020623989.
\end{align*}
\end{ex}

\begin{ex}
We consider series given by a $\gam^{\circ d}$-$\ga$-block chain with $\ga$-block having weight one.
For all $d\ge 1$, we have by \eqref{equ-gam-d=1},  \eqref{endga-block-Depth1} and \eqref{equ-binnSquare-gam1}
\begin{align*}
U_d:=&\, \sum_{n_1>\cdots>n_d>m> 0} \frac{a_{n_1}^2}{(2n_1-1)\cdots(2n_d-1)(2m)}\\
=&\, \frac{2}{\pi}  \int_0^{\pi/2} \cos t\,dt\, (\tan t\, dt)^{d-1} \,d(\sec t) \Big(\csc t-\cot t\Big) \,dt
= U_{d-1}-D_{d-1}=U_0-\sum_{k=0}^{d-1}D_k,
\end{align*}
where by \eqref{equ-anSqOver2n}
\begin{align*}
U_0:=&\, \frac{2}{\pi}\Big(\pi\log2-2G\Big),\\
D_0=&\, \frac{2}{\pi}  \int_0^{\pi/2} \cos t\,dt\,\csc t(\sec t-1) \,dt
= \frac{2}{\pi}  \int_0^{\pi/2}(1-\sin t) \csc t(\sec t-1) \,dt \\
=&\, \frac{2}{\pi}  \int_0^{\pi/2} (\csc t\sec t-\csc t-\sec t+1) \,dt
=  \frac{2}{\pi}  \Big(\frac{\pi}{2}+\int_0^1 2\tx_{-1} \Big)=  \frac{2}{\pi} \Big(\frac{\pi}{2}-2\log 2\Big)
\end{align*}
and all $d\ge 1$
\begin{align*}
D_d =&\, \frac{2}{\pi}  \int_0^{\pi/2}\bigg(\cos t\,dt\, (\tan t\, dt)^d \csc t(\sec t-1) \,dt\bigg)\\
=&\, \frac{2}{\pi}  \int_0^{\pi/2}\bigg((1-\sin t)\tan t \,dt\, (\tan t\, dt)^{d-1} \csc t(\sec t-1) \,dt\bigg)\\
=&\, D_{d-1}+\frac{2}{\pi}  \int_0^{\pi/2}\bigg((\tan t-\sec t) \,dt\, (\tan t\, dt)^{d-1} \csc t(\sec t-1) \,dt\bigg)\\
=&\, D_{d-1}-(-1)^d \frac{2}{\pi}  \int_0^1 (\ta+2\tx_{-1})\tz^{d-1} (\tx_{-i}+\tx_{i}) \\
=&\, D_0+ \frac{2}{\pi} \sum_{j=0}^{d-1}  (-1)^j\int_0^1 (\ta+2\tx_{-1})\tz^j (\tx_{-i}+\tx_{i}).
\end{align*}
Hence for all $d\ge 1$,
\begin{align*}
U_d:= U_0-d D_0-\frac{2}{\pi} \sum_{j=0}^{d-2}(-1)^j(d-1-j) \int_0^1 (\ta+2\tx_{-1})\tz^j (\tx_{-i}+\tx_{i}).
\end{align*}
Thus the weight drops by 1, as predicted by Thm.~\ref{thm-binnSquare}\ref{enu:thm-binnSquare-gen},
except for the special case when $d=1$ when there is no weight drop.
In particular,
\begin{align*}
U_1=&\, \frac{2}{\pi}\Big(\pi\log2-2G-\frac{\pi}{2}+2\log 2)\Big)\\
U_2=&\, \frac{2}{\pi}\Big(\pi\log2-2G-\pi+4\log 2 -\int_0^1 (\ta+2\tx_{-1}) (\tx_{-i}+\tx_{i})\Big)\\ 
=&\,\frac{2}{\pi}\Big(\pi\log2-2G-\pi+4\log 2+\frac{\pi^2}{12}-\frac32 \log^2 2\Big)\approx 0.049935489. 
\end{align*}
Note that $U_1$ is one of the two reasons we need to introduce $\iota(l)$ in Thm.~\ref{thm-binnSquare}
since there is no weight drop in this special case.
\end{ex}

\begin{ex} \label{eg-binn-sqr-gam-ga-block2}
We consider series given by a $\gam$-$\ga$-block chain with $\ga$-block having weight one. When $\gam$-block has weight one
there is no weight drop from the last example.
But with higher weight $\gam$-block, the weight drop pattern resumes:
\begin{align*}
&\, \sum_{n_1>n_2> 0} \frac{a_{n_1}^2}{(2n_1-1)^2(2n_2)}
=\frac{2}{\pi}  \int_0^{\pi/2} \sin t\, dt  \cot^2 t\,dt\,d(\sec t) \Big(\csc t-\cot t\Big) \,dt\\
=&\,\frac{2}{\pi}  \int_0^{\pi/2} \bigg(- \sin t\, dt \, dt\,  d(\sec t)  \Big(\csc t-\cot t\Big) \,dt\\
&\, -\cos  t\, dt \,d(\sec t)  \Big(\csc t-\cot t\Big) \,dt +\sin t\, dt \,\sec t\, dt  \Big(\csc t-\cot t\Big) \,dt\bigg)\\
=&\,\frac{2}{\pi}  \int_0^{\pi/2} \bigg(-2\cos  t\, dt \,d(\sec t)  \Big(\csc t-\cot t\Big) \,dt + dt  \Big(\csc t-\cot t\Big) \,dt\bigg)\\
=&\,\frac{2}{\pi}  \int_0^{\pi/2} \bigg(  - dt \Big(\csc t-\cot t\Big) \,dt
    +2\cos  t\, dt  \Big(\csc t\sec t -\csc t\Big)\,dt  \bigg)\\
=&\,\frac{2}{\pi}  \int_0^{\pi/2} \bigg( -dt \Big(\csc t-\cot t\Big) \,dt
 + 2\Big(1-\sin t\Big)\Big(\csc t\sec t -\csc t\Big)\,dt  \bigg)\\
=&\,\frac{2}{\pi}  \int_0^{\pi/2} \bigg( -dt \Big(\csc t-\cot t\Big) \,dt
 + 2\Big(\csc t\sec t -\csc t-\sec t+1\Big)\,dt \bigg)\\
=&\,\frac{2}{\pi} \Big(  \pi -\int_0^1 i(2\tx_{-1}-\tx_{-i}-\tx_{i})\td_{-i,i}-4\tx_{-1}  \Big)\\
=&\,\frac{2}{\pi}\Big(2G-\pi\log2 +\pi - 4\log 2\Big) \approx 0.01486445.
\end{align*}
\end{ex}

\begin{ex}\label{eg-V1}
We consider series given by a $\gam^{\circ d}$-$\ga$-block chain. We have by \eqref{endga-block}
and \eqref{equ-binnSquare-gam1}
\begin{align*}
V_1(s):=&\, \sum_{n>m> 0} \frac{a_{n}^2}{(2n-1)(2m)^s}
=\frac{2}{\pi}  \int_0^{\pi/2} \cos t\,dt\, \,d(\sec t) (\cot t \,dt)^{s-1} \Big(\csc t-\cot t\Big) \,dt  \\
=&\, \frac{2}{\pi}  \int_0^{\pi/2}  dt\, (\cot t \,dt)^{s-1} \Big(\csc t-\cot t\Big) \,dt
-(1-\sin t) \csc t \,dt  (\cot t \,dt)^{s-2} \Big(\csc t-\cot t\Big)  \,dt  \\
=&\, \frac{2(-1)^s}{\pi}  \int_0^1   (2\tx_{-1}-\tx_{-i}-\tx_{i})\ty^{s-2}\Big(-i\ty\td_{-i,i}-\td_{-1,1}+i\td_{-i,i}\Big).
\end{align*}
Hence
\begin{align}\label{equ-V1}
V_1:=V_1(2)=&\,\frac{2}{\pi} \left(\frac{3\pi^3}{16}- \frac{\pi\log 2}{4} (3\log 2-4)-
 8\Im\Li_3\Big(\frac{1+i}2\Big)-2 G -\frac{\pi^2}{24}\right)  \approx 0.03529309, \\
V_1(3)=&\,24\beta(4) + \frac1{16}\Big(\pi^3 (3 - 6 \log2)+ 2\zeta(3)\Big) +  \frac{\pi}4 \Big( 2 \log^32-3 \log^22 + \zeta(3)\Big) ,\nonumber\\
&\, -8\Im\Li_{3}\Big(\frac{1+i}2\Big) -16 \Im\Li_{4}\Big(\frac{1+i}2\Big)  \approx 0.024452426383. \label{equ-V1s=3}
\end{align}
For any $s\ge 2$, we have by  \eqref{equ-gam-d=1},  \eqref{endga-block} and \eqref{equ-binnSquare-gam1}
\begin{align*}
V_d=&\, \sum_{n_1>\cdots>n_d>m> 0} \frac{a_{n_1}^2}{(2n_1-1)\cdots(2n_d-1)(2m)^2} \\
=&\,\frac{2}{\pi}  \int_0^{\pi/2} \cos t\,dt\,(\tan t\,dt)^{d-1} \,d(\sec t)  \cot t \,dt \Big(\csc t-\cot t\Big) \,dt
= V_{d-1}- \frac{2}{\pi} L_s=V_1- \frac{2}{\pi} \sum_{k=2}^s L_k,
\end{align*}
where
\begin{align*}
L_1:= \int_0^{\pi/2}\cos t\,dt\, \csc t \,dt   \Big(\csc t-\cot t\Big)  \,dt
=&\,  \int_0^{\pi/2} (1-\sin t) \csc t \,dt   \Big(\csc t-\cot t\Big)  \,dt \\
= &\,\int_0^1(2\tx_{-1}-\tx_{-i}-\tx_{i})(\td_{-1,1}- i\td_{-i,i}),
\end{align*}
and for all $s\ge 2$
\begin{align*}
L_d:=&\, \int_0^{\pi/2}\cos t\,dt\,(\tan t\,dt)^{d-1} \csc t \,dt   \Big(\csc t-\cot t\Big)  \,dt \\
= &\,\int_0^{\pi/2} (1-\sin t)\tan t\,dt \,(\tan t\,dt)^{d-2}  \csc t \,dt   \Big(\csc t-\cot t\Big)  \,dt \\
= &\,L_{s-1}+\int_0^{\pi/2} (\tan t-\sec t)\,dt \,(\tan t\,dt)^{d-2}  \csc t \,dt   \Big(\csc t-\cot t\Big)  \,dt \\
= &\,L_{s-1}+(-1)^s \int_0^1(2\tx_{-1}-\tx_{-i}-\tx_{i})\td_{-1,1}\tz^{d-2}(\tx_{-i}+\tx_{i}) \\
= &\,L_1+  \sum_{j=0}^{d-2} (-1)^j  \int_0^1(2\tx_{-1}-\tx_{-i}-\tx_{i})\td_{-1,1}\tz^j (\tx_{-i}+\tx_{i}) .
\end{align*}
Hence
\begin{align*}
V_d= &\, V_1- \frac{2}{\pi} \left( (d-1) L_1+ \sum_{j=0}^{d-2} (s-1-j) (-1)^j  \int_0^1(2\tx_{-1}-\tx_{-i}-\tx_{i})\td_{-1,1}\tz^j (\tx_{-i}+\tx_{i}) \right).
\end{align*}
So there is no weight drop in this case. In particular,
\begin{align*}
V_2=&\,\frac{2}{\pi} \left(\frac{3\pi^3}{16}-G(\pi+4)- 8\Im\Li_3\Big(\frac{1+i}2\Big)+ \frac{\pi^2}{24} (\log 2-2)    - \frac{\pi\log 2}{4} (3\log 2-8)+\frac{35}{16}\ze(3) \right)\\
 \approx &\, 0.01707012767.
\end{align*}
\end{ex}

\begin{cor} We have
\begin{align*}
\sum_{n=1}^\infty \frac{a_n^2 H_n^{(3)}}{2n-1}=\frac{2}{\pi}\Big(16 G +4\pi - 8\pi\log2- \zeta(3)-8\Big).
\end{align*}
\end{cor}
\begin{proof} By partial fractions we have
\begin{align*}
\frac1{(2n-1)(2n)^3}=\frac1{2n-1}-\frac1{(2n)^3}-\frac1{(2n)^2}-\frac1{2n}.
\end{align*}
Thus by \eqref{equ-Gamma1} and \eqref{equ-V1s=3}
\begin{align*}
\frac18\sum_{n=1}^\infty \frac{a_n^2 H_n^{(3)}}{2n-1}
=&\, \sum_{n=1}^\infty \frac{a_n^2}{(2n-1)(2n)^3}+\sum_{n>m>0}\frac{a_n^2}{(2n-1)(2m)^3}\\
=&\,\frac{2}{\pi} \Big(\frac{\pi}{2}-1\Big) +V_1(3)
-\sum_{n=1}^\infty \frac{a_n^2}{(2n)^3}-\sum_{n=1}^\infty \frac{a_n^2}{(2n)^2}-\sum_{n=1}^\infty \frac{a_n^2}{2n}
\end{align*}
Now the last three sums are given in Example \ref{eg-binn-sqr-single-ga-block}. The corollary follows immediately.
\end{proof}

\begin{ex}\label{eg-W1}
We consider series given by a $\gam$-$\ga$-block chain with $\ga$-block having weight two. We have
$W_1:=V_1$, and by  \eqref{endga-block} and \eqref{equ-binnSquare-gam-depthd}
\begin{align}\label{equ-W2}
W_2:=&\, \sum_{n>m> 0} \frac{a_{n}^2}{(2n-1)^2(2m)^2}
=\frac{2}{\pi}  \int_0^{\pi/2} \sin t\,dt\,\cot^2 t\,dt\,d(\sec t)  \cot t \,dt \Big(\csc t-\cot t\Big) \,dt \nonumber\\
=&\, -W_1+\int_0^{\pi/2}  \sin t\,dt\, (\sec t \,dt- dt\,d(\sec t) )  \cot t \,dt \Big(\csc t-\cot t\Big) \,dt   \nonumber\\
=&\, -W_1+\int_0^{\pi/2}  \cos t\,dt\,\csc t \,dt \Big(\csc t-\cot t\Big) \,dt   \nonumber\\
=&\, -W_1+\int_0^{\pi/2}  (1-\sin t)\csc t \,dt \Big(\csc t-\cot t\Big) \,dt \nonumber\\
=&\, \frac{2}{\pi}  \int_0^1 i (2\tx_{-1}-\tx_{-i}-\tx_{i})\ty \td_{-i,i}
+2(2\tx_{-1}-\tx_{-i}-\tx_{i})\td_{-1,1} - 2i (2\tx_{-1}-\tx_{-i}-\tx_{i}) \td_{-i,i}\nonumber\\
=&\,\frac{2}{\pi} \left( 4G+ \frac{\pi^2}{12} -\frac{3\pi^3}{16}+8\Im\Li_3\Big(\frac{1+i}2\Big)+\frac{\pi\log 2}{4} (3\log 2-8) \right)
 \approx 0.006455549.
\end{align}
For all $s\ge3$,
\begin{align*}
W_s:=&\, \sum_{n>m> 0} \frac{a_{n}^2}{(2n-1)^s(2m)^2}
=\frac{2}{\pi}  \int_0^{\pi/2} \sin t\,dt\,(\cot t\,dt)^{s-2} \cot^2 t\,dt\,d(\sec t)  \cot t \,dt \Big(\csc t-\cot t\Big) \,dt \\
=&\, -W_{s-1}+\frac{2}{\pi}\int_0^{\pi/2} \sin t\,dt\,(\cot t\,dt)^{s-2} (\sec t\,dt-dt\,d(\sec t))  \cot t \,dt \Big(\csc t-\cot t\Big) \,dt \\
=&\, -W_{s-1}+\frac{2}{\pi}N_{s-2} =(-1)^s W_2 +\frac{2}{\pi} \sum_{k=1}^{s-2} (-1)^{s-k} N_k,
\end{align*}
where
\begin{align*}
N_0=&\,  \int_0^{\pi/2} \sin t\,dt\,dt \, \csc t \,dt \Big(\csc t-\cot t\Big) \,dt
= \int_0^{\pi/2}  \cos t\,dt  \csc t \,dt\Big(\csc t-\cot t\Big) \,dt \\
=&\,  \int_0^{\pi/2}  (1-\sin t)\csc t \,dt\Big(\csc t-\cot t\Big) \,dt
=  \int_0^1(2\tx_{-1}-\tx_{-i}-\tx_{i})(  \td_{-1,1}-i\td_{-i,i})
\end{align*}
and for all $s\ge 1$
\begin{align*}
N_s=&\,  \int_0^{\pi/2} \sin t\,dt \,dt\,(\cot t\,dt)^s \, dt\,\csc t \,dt \Big(\csc t-\cot t\Big) \,dt \\
=&\,  \int_0^{\pi/2} \cos t\cot t\,dt \,(\cot t\,dt)^{s-1} dt\,\csc t \,dt \Big(\csc t-\cot t\Big) \,dt \\
=&\, -N_{s-1}+\int_0^{\pi/2} \csc t\,dt \,(\cot t\,dt)^{s-1} dt\,\csc t \,dt \Big(\csc t-\cot t\Big) \,dt \\
=&\, -N_{s-1}-(-1)^s i \int_0^1(2\tx_{-1}-\tx_{-i}-\tx_{i})\td_{-1,1}\td_{-i,i}\ty^{s-1} \td_{-1,1} \\
=&\, (-1)^s N_0-(-1)^s\sum_{j=0}^{s-1} i \int_0^1(2\tx_{-1}-\tx_{-i}-\tx_{i})\td_{-1,1}\td_{-i,i}\ty^j \td_{-1,1}.
\end{align*}
Hence
\begin{align*}
W_s= &\,(-1)^s W_2 +\frac{2}{\pi} \sum_{k=1}^{s-2} (-1)^{s-k} N_k\\
= &\,(-1)^s W_2 +(-1)^s\frac{2}{\pi}\left( (s-2)N_0
-\sum_{j=0}^{s-3}   (s-2-j) i \int_0^1(2\tx_{-1}-\tx_{-i}-\tx_{i})\td_{-1,1}\td_{-i,i}\ty^j \td_{-1,1}\right).
\end{align*}
In particular,
\begin{align*}
W_3=&\,-W_2-\frac{2}{\pi} \left(N_0-i\int_0^1(2\tx_{-1}-\tx_{-i}-\tx_{i})\td_{-1,1}\td_{-i,i}\td_{-1,1}\right) \\
=&\,  \frac{2}{\pi}\left(6\gb(4)+\frac{3\pi^3}{16} - 8\Im\Li_3\Big(\frac{1+i}2\Big)-6G+ \frac{\pi^2}{24} (2G-3)
 - \frac{3\pi\log^2 2}{4} +\log^3 2 -\frac{7}{4}\ze(3) \right)
 \approx 0.001685043.
\end{align*}
\end{ex}

\begin{ex}\label{eg-Y1}
We consider series given by a $\gam$-$\ga$-$\ga$-block chain with both $\ga$-blocks having weight one. We have
by  \eqref{midgaga-block}, \eqref{endga-block} and \eqref{equ-binnSquare-gam1}
\begin{align*}
Y_1:=&\, \sum_{n>k>m> 0} \frac{a_{n}^2}{(2n-1)(2k)(2m)} \\
=&\, \frac{2}{\pi}  \int_0^{\pi/2} \cos t\,dt \,d(\sec t) \Big(\csc t \,dt \circ\sec t-\cot t \,dt\Big) \,\Big(\csc t-\cot t\Big) \,dt  \nonumber\\
=&\, \frac{2}{\pi}  \int_0^{\pi/2}\bigg( dt \,\csc t \,dt \Big(\sec t\csc t-\csc t\Big)\,dt
    - dt \,\cot t \,dt \, \Big(\csc t-\cot t\Big)\,dt    \nonumber\\
 &\,-\cos t\,dt \Big(\sec t \csc t\,dt \circ\sec t- \csc t\, dt\Big) \,\Big(\csc t-\cot t\Big) \,dt  \bigg) \nonumber\\
=&\, \frac{2}{\pi}  \int_0^{\pi/2}\bigg( dt \,\csc t \,dt \Big(\sec t\csc t-\csc t\Big)\,dt
    - dt \,\cot t \,dt \, \Big(\csc t-\cot t\Big)\,dt     \nonumber\\
 &\,-\Big(1-\sin t\Big)\Big(\sec t \csc t\,dt \circ\sec t- \csc t\, dt\Big) \,\Big(\csc t-\cot t\Big) \,dt   \bigg)\nonumber\\
=&\, \frac{2}{\pi}  \int_0^{\pi/2}\bigg( dt \,\csc t \,dt \Big(\sec t\csc t-\csc t\Big)\,dt
    - dt \,\cot t \,dt \, \Big(\csc t-\cot t\Big)\,dt    \nonumber\\
 &\,   -\Big(\sec t \csc t-\sec t\Big)\,dt\,\Big(\sec t\csc t-\csc t\Big) \,dt +\Big(\csc t-1\Big)\,dt\,\Big(\csc t-\cot t\Big) \,dt   \bigg)\nonumber\\
=&\, \frac{2}{\pi}\int_0^1   \bigg( i(\ta+2\tx_{-1})\td_{-1,1}\td_{-i,i}+i(2\tx_{-1}-\tx_{-i}-\tx_{i})\ty\td_{-i,i}
-(\ta+2\tx_{-1})(\tx_{-1}+\tx_{1}) \nonumber\\
 &\, \hskip2cm
 +(2\tx_{-1}-\tx_{-i}-\tx_{i})(\td_{-1,1} -i\td_{-i,i}) \bigg) \nonumber\\
=&\,\frac{2}{\pi} \left(
16\Im\Li_3\Big(\frac{1+i}2\Big)- \frac{19\pi^3}{48}  +
 \frac{\pi^2}8+ \frac{\pi}2\log2(\log2-2) - 2\log^2 2 + G(4\log 2+2) \right)  \\
  \approx &\, 0.0441495729. \nonumber
\end{align*}
Similarly,
\begin{align*}
Y_2:=&\, \sum_{n>k>m> 0} \frac{a_{n}^2}{(2n-1)^2(2k)(2m)} \nonumber\\
=&\, \frac{2}{\pi}  \int_0^{\pi/2} \sin t\,dt \,\cot^2 t\,dt \,d(\sec t) \Big(\csc t \,dt \circ\sec t-\cot t \,dt\Big) \,\Big(\csc t-\cot t\Big) \,dt  \nonumber\\
=&\, \frac{2}{\pi}  \int_0^{\pi/2}\bigg(-\sin t\,dt\,dt \,d(\sec t) \Big(\csc t \,dt \circ\sec t-\cot t \,dt\Big) \,\Big(\csc t-\cot t\Big) \,dt    \nonumber\\
 &\,-\cos t\,dt \,d(\sec t) \Big(\csc t \,dt \circ\sec t-\cot t \,dt\Big) \,\Big(\csc t-\cot t\Big) \,dt  \nonumber\\
 &\,+ \sin t\,dt \,\sec t\,dt \Big(\csc t \,dt \circ\sec t-\cot t \,dt\Big) \,\Big(\csc t-\cot t\Big) \,dt \nonumber\\
=&\,-Y_1+ \frac{2}{\pi}  \int_0^{\pi/2}\bigg(-dt \Big(\csc t \,dt \circ\sec t-\cot t \,dt\Big) \,\Big(\csc t-\cot t\Big) \,dt  \nonumber\\
&\,+\cos t\,dt  \Big(\sec t\csc t \,dt \circ\sec t-\csc t \,dt\Big) \,\Big(\csc t-\cot t\Big) \,dt  \nonumber\\
 &\,+  dt \Big(\csc t \,dt \circ\sec t-\cot t \,dt\Big) \,\Big(\csc t-\cot t\Big) \,dt \nonumber\\
=&\,-Y_1+ \frac{2}{\pi}  \int_0^{\pi/2}\bigg(\Big(1-\sin t\Big) \Big(\sec t\csc t \,dt \circ\sec t-\csc t \,dt\Big) \,\Big(\csc t-\cot t\Big) \,dt  \nonumber\\
=&\,-Y_1+ \frac{2}{\pi}  \int_0^{\pi/2}\bigg( \Big(\sec t(\csc t-1)\,dt \circ\sec t+(1-\csc t )\,dt\Big) \,\Big(\csc t-\cot t\Big) \,dt  \nonumber\\
=&\, -Y_1+\frac{2}{\pi}\int_0^1   \bigg( (\ta+2\tx_{-1})(\tx_{-1}+\tx_{1})
 +(2\tx_{-1}-\tx_{-i}-\tx_{i})(i\td_{-i,i} -\td_{-1,1}) \bigg) \nonumber\\
=&\,\frac{2}{\pi} \left(
-16\Im\Li_3\Big(\frac{1+i}2\Big)+ \frac{19\pi^3}{48}-
 \frac{\pi^2}4+ \frac{\pi}2(4\log2-\log^2 2)   + 4\log^2 2 - 4G(\log 2+2) \right) \\
 \approx  &\,0.002234785. \nonumber
\end{align*}
For all $s\ge3$, by \eqref{endgb-block} and \eqref{equ-binnSquare-gam-depthd}
\begin{align*}
Y_s:=&\, \sum_{n>k>m> 0} \frac{a_{n}^2}{(2n-1)^s(2k)(2m)} \\
=&\, \frac{2}{\pi}  \int_0^{\pi/2} \sin t\,dt\,(\cot t\,dt)^{s-2} \cot^2 t\,dt\,d(\sec t)  \Big(\csc t \,dt \circ\sec t-\cot t \,dt\Big) \,\Big(\csc t-\cot t\Big) \,dt  \\
=&\, -Y_{s-1}+\frac{2}{\pi} \int_0^{\pi/2}  \sin t\,dt\,(\cot t\,dt)^{s-2} \Big(\sec t\,dt-dt\,d(\sec t)\Big)  \Big(\csc t \,dt \circ\sec t-\cot t \,dt\Big) \Big(\csc t-\cot t\Big) \,dt \\
=&\, -Y_{s-1}+\frac{2}{\pi} \int_0^{\pi/2}  \sin t\,dt\,(\cot t\,dt)^{s-2}\,dt\, \Big(\sec t\csc t \,dt \circ\sec t-\csc t \,dt\Big) \Big(\csc t-\cot t\Big) \,dt \\
=&\, -Y_{s-1}+\frac{2}{\pi}M_{s-2} =(-1)^s Y_2 +\frac{2}{\pi} \sum_{k=1}^{s-2} (-1)^{s-k} M_k,
\end{align*}
where
\begin{align*}
M_0=&\,  \int_0^{\pi/2} \sin t\,dt\, \,dt\,\Big(\sec t\csc t \,dt \circ\sec t-\csc t \,dt\Big) \Big(\csc t-\cot t\Big) \,dt\\
=&\,  \int_0^{\pi/2} \cos t\,dt\,\Big(\sec t\csc t \,dt \circ\sec t-\csc t \,dt\Big) \Big(\csc t-\cot t\Big) \,dt\\
=&\,  \int_0^{\pi/2} \Big(1-\sin t\Big) \Big(\sec t\csc t \,dt \circ\sec t-\csc t \,dt\Big) \Big(\csc t-\cot t\Big) \,dt\\
= &\,  \int_0^{\pi/2}   \Big((\sec t\csc t -\sec t) \,dt \circ\sec t+(1-\csc t) \,dt\Big) \Big(\csc t-\cot t\Big) \,dt \\
=&\,  \int_0^1 (\ta+2\tx_{-1})(\tx_{-1}+\tx_{1})+(2\tx_{-1}-\tx_{-i}-\tx_{i}) (i\td_{-i,i}-\td_{-1,1})
\end{align*}
and for all $s\ge 1$
\begin{align*}
M_s=&\, \int_0^{\pi/2} \sin t\,dt\,(\cot t\,dt)^s \,dt\, \Big(\sec t\csc t \,dt \circ\sec t-\csc t \,dt\Big) \Big(\csc t-\cot t\Big) \,dt \\
=&\,  \int_0^{\pi/2}\cos t\cot t\,dt (\cot t\,dt)^{s-1} \,dt\, \Big(\sec t\csc t \,dt \circ\sec t-\csc t \,dt\Big) \Big(\csc t-\cot t\Big) \,dt \\
=&\, -M_{s-1}+\int_0^{\pi/2} \csc t\,dt \,(\cot t\,dt)^{s-1} dt\,dt\,\Big(\sec t\csc t \,dt \circ\sec t-\csc t \,dt\Big)\Big(\csc t-\cot t\Big) \,dt \\
=&\, -M_{s-1}+(-1)^s i\int_0^1 \Big( (\ta+2\tx_{-1})(\ta+\tx_{-1}+\tx_{1})-(2\tx_{-1}-\tx_{-i}-\tx_{i})\td_{-1,1} \Big) \td_{-i,i}\ty^{s-1} \td_{-1,1}\\
=&\, (-1)^s M_0-(-1)^s \sum_{j=0}^{s-1} \int_0^1 \Big( (\ta+2\tx_{-1})(\ta+\tx_{-1}+\tx_{1})-(2\tx_{-1}-\tx_{-i}-\tx_{i})\td_{-1,1} \Big)\td_{-i,i}\ty^j \td_{-1,1}.
\end{align*}
Hence
\begin{align*}
Y_s= &\,(-1)^s Y_2 +\frac{2}{\pi} \sum_{k=1}^{s-2} (-1)^{s-k} M_k
=  (-1)^s Y_2 +(-1)^s\frac{2}{\pi}\bigg( (s-2)M_0     \\
 &\, \hskip1cm -\sum_{j=0}^{s-3}   (s-2-j) i \int_0^1\Big( (\ta+2\tx_{-1})(\ta+\tx_{-1}+\tx_{1})-(2\tx_{-1}-\tx_{-i}-\tx_{i})\td_{-1,1} \Big)\td_{-i,i}\ty^j \td_{-1,1} \bigg).
\end{align*}
In particular,
\begin{align*}
Y_3=&\,-Y_2-\frac{2}{\pi} \left(M_0-\int_0^1\Big( (\ta+2\tx_{-1})(\ta+\tx_{-1}+\tx_{1})-(2\tx_{-1}-\tx_{-i}-\tx_{i})\td_{-1,1} \Big) \td_{-i,i} \td_{-1,1} \right)
\\
=&\,  \frac{2}{\pi}\bigg( \frac{3\pi^2}{8} - 10 \gb(4) + 16 \Im\Li_3\Big(\frac{1+i}2\Big)
    + 8\Im\Li_4\Big(\frac{1+i}2\Big)+ 8\Im\Li_3\Big(\frac{1+i}2\Big)\log 2 \\
&\,  -\frac{3\pi^3}{96} (38 + 15\log 2)  - 6\log^2 2 +G\Big(6 - \frac{\pi^2}{4}+ 4 \log^2 2 + 4\log 2\Big) \\
&\, + \pi\Big(\frac12\log^2 2 -  \frac{5\log^3 2}{24}  - 3\log 2 +\frac{7}{4}\ze(3) \Big) \bigg)
 \approx 0.000241778756.
\end{align*}
\end{ex}

\begin{ex}
We consider series given by a $\gam^{\circ d}$-$\gb^\star$-block chain, with all $\gam$-blocks having
weight one. When $d=1$, by \eqref{endgb-block-Depth1} and \eqref{equ-binnSquare-gam1}
\begin{align*}
X_1:=&\, \sum_{n>m\ge 0} \frac{a_{n}^2}{(2n-1)(2m+1)}
=\frac{2}{\pi}  \int_0^{\pi/2} \cos t\,dt \,d(\tan t) \,dt\\
=&\, \frac{2}{\pi}  \int_0^{\pi/2}\bigg( \sin t\,dt\,dt -\cos t\,dt\, \tan t\, dt\bigg)
=  \frac{2}{\pi}  \int_0^{\pi/2}  \Big( \cos t\,dt -(1-\sin t)\, \tan t\, dt \Big)  \\
=&\, \frac{2}{\pi}  \int_0^{\pi/2}   (\sec t-\tan t) \, dt
= \frac{-2}{\pi} \Big(\int_0^1  (\tx_{-i}+\tx_{i}) \Big)
= \frac{2\log 2}{\pi}  .
\end{align*}
Suppose $d\ge 2$. Then by \eqref{equ-gam-d=1},  \eqref{endgb-block-Depth1} and \eqref{equ-binnSquare-gam1}
\begin{align*}
X_d:=&\, \sum_{n_1>\cdots>n_d>m\ge 0} \frac{a_{n_1}^2}{(2n_1-1)\cdots (2n_d-1)(2m+1)}
=\frac{2}{\pi}  \int_0^{\pi/2} \cos t\,dt\, (\tan t\, dt)^{d-1} \,d(\tan t) \,dt\\
=&\, \frac{2}{\pi}  \int_0^{\pi/2}\bigg( \cos t\,dt\, (\tan t\, dt)^{d-2} (\sec^2t-1) \,dt \,dt -\cos t\,dt\, (\tan t\, dt)^{d} \bigg)\\
=&\,  X_{d-1}-\frac{2}{\pi} (E_{d-2}+A_d),
\end{align*}
where $A_d$ is defined by \eqref{defn-Ad}, $E_0=\pi^2/8-1$ and for all $d\ge1$
\begin{align*}
E_d=&\,  \int_0^{\pi/2} \cos t\,dt\, (\tan t\, dt)^{d} \,dt\,dt
=  \int_0^{\pi/2}  (1-\sin t)\tan t\, dt\, (\tan t\, dt)^{d-1} \,dt \,dt  \\
=&\,E_{d-1}+  \int_0^{\pi/2}  (\tan t-\sec t)\, dt\, (\tan t\, dt)^{d-1} \,dt\,dt
= E_{d-1}+(-1)^d   \int_0^1 \td_{-i,i}^2 \tz^{d-1} (\tx_{-i}+\tx_{i})  \\
=&\, E_0-\sum_{j=0}^{d-1} (-1)^{j} \int_0^1 \td_{-i,i}^2 \tz^j (\tx_{-i}+\tx_{i}),
\end{align*}
Thus for all $d\ge 2$, by \eqref{defn-Ad}
\begin{align*}
X_d=&\,X_1- \frac{2}{\pi} \left(\sum_{k= 0}^{d-2}  E_k+\sum_{k=2}^{d} A_k \right)
=\frac{2}{\pi} \bigg(d\log(2) -(d-1)\frac{\pi^2}{8}\\
&\, +\sum_{j=0}^{d-3} (-1)^{j}(d-2-j) \int_0^1   \td_{-i,i}^2 \tz^j (\tx_{-i}+\tx_{i})
- \sum_{j=1}^{d-1} (-1)^{j}(d-j) \int_0^1 \tz^j (\tx_{-i}+\tx_{i})   \bigg).
\end{align*}
We see the weight drops by 1, as predicted by Thm.~\ref{thm-binnSquare}\ref{enu:thm-binnSquare-gen}.
In particular,
\begin{align*}
X_2=\frac{2}{\pi}\left( 2\log 2-\frac{\pi^2}{12}-\frac{1}{2} \log^2 2\right) \approx 0.20601068.
\end{align*}
\end{ex}


\begin{thebibliography}{99}

\bibitem{Akhilesh1}
P.\ Akhilesh, Double tails of multiple zeta values, \emph{J. Number Thy.} \textbf{170} (2017), pp.\ 228--249.

\bibitem{Akhilesh}
P.\ Akhilesh, Multiple zeta values and multiple Ap\'ery-like sums, \emph{J. Number Thy.} \textbf{226} (2021), pp.\ 72--138.

\bibitem{Au2020}
K.C. Au, Evaluation of one-dimensional polylogarithmic integral, with applications to infinite series, arXiv:2007.03957. A companion Mathematica package available at researchgate.net/publication/357601353.


\bibitem{BognerLu2013}
C.\ Bogner and M,\ L\"uders, Multiple polylogarithms and linearly reducible Feynman graphs,
in: Feynman Amplitudes, Periods and Motives, eds.\ L. Alvarez-consul, J.\ I.\ Burgos-gil
and K.~Ebrahimi-Fard, \emph{Contemp. Math.} \textbf{648}, pp.\  11--28, AMS, 2015.

\bibitem{Broadhurst1996}
D.J.\ Broadhurst, Conjectured enumeration of
irreducible multiple zeta values, from knots and Feynman diagrams. arXiv:hep-th/9612012.

\bibitem{Broadhurst1999}
D.J.\ Broadhurst, Massive 3-loop Feynman diagrams reducible to SC$^*$
primitives of algebras of the sixth root of unity, \emph{European
Phys.\ J.\ C (Fields)} \textbf{8} (1999), pp.\ 311--333.

\bibitem{BroadhurstKr1997}
D.J.\ Broadhurst and D.\ Kreimer,
Association of multiple zeta values with positive knots via Feynman
diagrams up to 9 loops
\emph{Phys.\ Lett.\ B} \textbf{393} (1997), pp.\ 403--412.

\bibitem{Brown2011}
F.\ Brown, Multiple zeta values and periods: from moduli spaces to Feynman integrals,
in: Combinatorics and physics, 27--52, \emph{Contemp.\ Math.}, Vol.\ \textbf{539},
Amer.\ Math.\ Soc., Providence, RI, 2011.

\bibitem{Campbell2018}
J.\ M.\ Campbell, Ramanujan-like series for $\frac{1}{\pi}$ involving harmonic numbers,
\emph{The Ramanujan J.} \textbf{46} (2018), pp.\ 373--387.

\bibitem{Campbell2019}
J.\ M.\ Campbell, Series containing squared central binomial coefficients and alternating harmonic
numbers,\emph{ Mediterr. J. Math.} \textbf{16} (2019) Art. 37, 7.

\bibitem{CampbellDS2019}
J.\ M.\ Campbell, J.\ D'Aurizio and J.\ Sondow,
On the interplay among hypergeometric functions, complete elliptic integrals, and Fourier--Legendre expansions,
\emph{J. Math. Anal. Appl.} \textbf{479} (2019), pp.\ 90--121.

\bibitem{CampbellSofo2017}
J.\ M.\ Campbell and A.\ Sofo,
An integral transform related to series involving alternating harmonic numbers,
\emph{Integral Transforms and Special Func.} \textbf{28} (2017), pp.\ 547--559.

\bibitem{Chen2016}
H.\ Chen, Interesting series associated with central binomial coefficients, Catalan numbers and harmonic
numbers, \emph{J. Integer Seq.} \textbf{19} (2016). Article 16.1.5.

\bibitem{KTChen1971}
K.-T.\ Chen, Algebras of iterated path integrals and fundamental groups,
\emph{Trans.\ Amer.\ Math.\ Soc.} \textbf{156} (1971), pp.\  359--379.

\bibitem{KTChen1977}
K.-T.\ Chen, Iterated path integrals, \emph{Bull.\ Amer.\ Math.\ Soc.} \textbf{83} (1977), pp.\ 831--879.

\bibitem{Duhr2015}
C.\ Duhr, Scattering amplitude, Feynman integrals and multiple polylogarithms,
in: Feynman Amplitudes, Periods and Motives, eds.\ L. Alvarez-consul, J.\ I.\ Burgos-gil
and K.~Ebrahimi-Fard, \emph{Contemp. Math.} \textbf{648}, pp.\  109--133, AMS, 2015.

\bibitem{Kreimer2015}
D.\ Kreimer, Quantum fields, periods and algebraic geometry,
in: Feynman Amplitudes, Periods and Motives, eds.\ L. Alvarez-consul, J.\ I.\ Burgos-gil
and K.~Ebrahimi-Fard, \emph{Contemp. Math.} \textbf{648}, pp.\  153--167, AMS, 2015.

\bibitem{Leshchiner}
D.\ Leshchiner, Some new identities for $\gz(k)$, \emph{J. Number Theory}, \textbf{13} (1981), pp.\ 355--362.

\bibitem{JegerlehnerKV2003}
F. Jegerlehner, M.Yu. Kalmykov and O. Veretin, $\overline{\rm{MS}}$ versus pole masses of gauge bosons II:
two-loop electroweak Fermion corrections,
\emph{Nucl. Phys.} \textbf{B658} (2003), pp.\ 49--112.

\bibitem{Panzer2014a}
E.\ Panzer, On hyperlogarithms and Feynman integrals with dive
rgences and many scales,
\emph{JHEP} (71) (2014), pp.\ 1403. arXiv:1401.4361.

\bibitem{ParkerSSV2015}
D.\ Parker, A.\ Scherlis, M.\ Spradlin and A.\ Volovich,
Hedgehog bases for $A_n$ cluster polylogarithms and an application to six-point amplitudes.
\emph{J. High Energ. Phys.} (2015), Article $\sharp$136.

\bibitem{Racinet2002}
G.\ Racinet, Doubles m\'elanges des polylogarithmes multiples aux
racines de l'unit\'e (in French),
\emph{Publ.\ Math.\ IHES} \textbf{95} (2002), pp.\ 185--231.

\bibitem{Todorov2014}
I.\ Todorov, Polylogarithms and multizeta values in massless Feynman amplitudes,
Lie Theory and Its Applications in Physics, 2014, Vol. 111.

\bibitem{XuZhao2021b}
C.\ Xu and J.\ Zhao, Ap\'{e}ry-type series and colored multiple zeta values, arXiv:2111.10998.

\bibitem{XuZhao2022a}
C.\ Xu and J.\ Zhao, Ap\'{e}ry-type series with summation indices of mixed parities and colored multiple zeta values, I, arXiv:2202.06195.

\bibitem{XuZhao2022c}
C.\ Xu and J.\ Zhao, Ap\'{e}ry-type series with summation indices of mixed parities and colored multiple zeta values, III, arXiv:2205.01000.

\bibitem{Zhao2016}
J. Zhao, \emph{Multiple zeta functions, multiple polylogarithms and their special values}, Series on Number
Theory and its Applications, Vol.~12, World Scientific Publishing Co. Pte. Ltd., Hackensack, NJ, 2016.

\end{thebibliography}
\end{document}